\documentclass[a4paper,11pt,twoside,reqno]{amsart}
\usepackage[top=3cm, bottom=3cm, left=3cm, right=3cm]{geometry}

\usepackage{amsmath,amssymb,amsthm,amsfonts,xcolor,array,url,mathtools,float}

\theoremstyle{definition}
\newtheorem{definition}{Definition}[section]
\newtheorem{remark}[definition]{Remark}
\newtheorem{myremark}{Remark}
\newtheorem{notation}[definition]{Notation}

\theoremstyle{plain}
\newtheorem*{conjecture*}{Conjecture}
\newtheorem{mytheorem}{Theorem}
\newtheorem{theorem}[definition]{Theorem}
\newtheorem{proposition}[definition]{Proposition}
\newtheorem{lemma}[definition]{Lemma}
\newtheorem{corollary}[definition]{Corollary}

\usepackage[shortlabels]{enumitem}
\setlist{leftmargin=0.8cm}
\setlist[enumerate]{label=\rm{(\roman*)}}

\usepackage[justification=centering,aboveskip=-3pt,belowskip=3pt]{caption}
\usepackage{ltablex,hhline,booktabs}
\newcolumntype{C}{>{\centering\arraybackslash}X}

\numberwithin{equation}{section}

\newcommand{\new}{\quad}

\renewcommand{\SS}{\textsection}

\newcommand{\A}{\mathcal{A}}

\newcommand{\C}{\mathcal{C}}
\newcommand{\M}{\mathcal{M}}
\renewcommand{\S}{\mathcal{S}}
\newcommand{\T}{\mathcal{T}}

\newcommand{\Ga}{\Gamma}
\newcommand{\Om}{\Omega}
\newcommand{\La}{\Lambda}
\renewcommand{\a}{\alpha}
\renewcommand{\b}{\beta}

\renewcommand{\d}{\delta}
\newcommand{\e}{\varepsilon}
\renewcommand{\i}{\iota}
\newcommand{\p}{\varphi}
\renewcommand{\r}{\rho}
\newcommand{\s}{\sigma}
\renewcommand{\l}{\lambda}
\renewcommand{\t}{\tau}
\renewcommand{\th}{\theta}

\DeclareSymbolFont{CMletters}{OML}{cmm}{m}{it}
\DeclareMathSymbol{\xi}{\mathord}{CMletters}{"18}

\newcommand{\oly}{\overline{y}}
\newcommand{\olV}{\overline{V}}

\newcommand{\Nat}{\mathbb{N}}
\newcommand{\F}{\mathbb{F}}
\newcommand{\K}{\overline{\F}_q}

\DeclareSymbolFont{Symbols}{OMS}{cmsy}{m}{n}
\DeclareMathSymbol{\emptyset}{\mathord}{Symbols}{"3B}

\renewcommand{\:}{\colon}

\renewcommand{\mod}[1]{\mathrm{ \ } (\mathrm{mod \ } #1)}

\newcommand{\rad}{\mathrm{rad}}

\newcommand{\<}{\langle}
\renewcommand{\>}{\rangle}

\renewcommand{\sp}{{:}}
\newcommand{\soc}{\mathrm{soc}}
\newcommand{\Aut}{\mathrm{Aut}}

\newcommand{\Out}{\mathrm{Out}}
\newcommand{\InnDiag}{\mathrm{Inndiag}}
\newcommand{\fix}{\mathrm{fix}}
\newcommand{\fpr}{\mathrm{fpr}}

\newcommand{\SL}{\mathrm{SL}}
\newcommand{\GL}{\mathrm{GL}}
\newcommand{\PSL}{\mathrm{PSL}}
\newcommand{\PGL}{\mathrm{PGL}}
\newcommand{\Sp}{\mathrm{Sp}}
\newcommand{\GSp}{\mathrm{GSp}}
\newcommand{\PSp}{\mathrm{PSp}}
\newcommand{\PGSp}{\mathrm{PGSp}}
\newcommand{\SO}{\mathrm{SO}}
\renewcommand{\O}{\mathrm{O}}
\newcommand{\GO}{\mathrm{GO}}

\newcommand{\GU}{\mathrm{GU}}
\newcommand{\PSU}{\mathrm{PSU}}

\begin{document}


\author{Scott Harper}
\address{S. Harper, School of Mathematics, University of Bristol, Bristol, BS8 1TW, UK}
\email{scott.harper@bristol.ac.uk}

\title[Uniform spread of symplectic and orthogonal groups]{On the uniform spread of almost simple \\ symplectic and orthogonal groups}

\date{\today}
\subjclass[2010]{Primary 20D06; Secondary 20E28, 20F05, 20P05}
\keywords{Classical groups; Maximal subgroups; Uniform spread}
  
\begin{abstract}
A group is $\frac{3}{2}$-generated if every non-identity element is contained in a \mbox{generating} pair. A conjecture of Breuer, Guralnick and Kantor from 2008 asserts that a finite group is $\frac{3}{2}$-generated if and only if every proper quotient of the group is cyclic, and recent work of Guralnick reduces this conjecture to almost simple groups. In this paper, we prove a stronger form of the conjecture for almost simple symplectic and odd-dimensional orthogonal groups. More generally, we study the uniform spread of these groups, obtaining lower bounds and related asymptotics. This builds on earlier work of Burness and Guest, who established the conjecture for almost simple linear groups.
\end{abstract}

\maketitle


\section{Introduction}\label{sec:Intro}
Let $G$ be a finite group. We say that $G$ is \emph{$d$-generated} if $G$ has a generating set of size $d$. It is well-known that every finite simple group is 2-generated \cite{ref:Steinberg62,ref:AschbacherGuralnick84}. In fact, almost surely, any two elements of a finite simple group $G$ generate the group, in the sense that the probability that two randomly chosen elements form a generating pair tends to one as $|G|$ tends to infinity \cite{ref:KantorLubotzky90,ref:LiebeckShalev95}. Therefore, generating pairs are abundant in finite simple groups, and it is natural to ask how they are distributed across the group. With this in mind, we say that $G$ is \emph{$\frac{3}{2}$-generated} if every non-identity element of $G$ is contained in a generating pair. By a theorem of Guralnick and Kantor \cite{ref:GuralnickKantor00} (also see Stein \cite{ref:Stein98}), every finite simple group is $\frac{3}{2}$-generated, resolving a question of Steinberg \cite{ref:Steinberg62} in the affirmative. \par

It is straightforward to see that every proper quotient of a $\frac{3}{2}$-generated group is cyclic. In \cite{ref:BreuerGuralnickKantor08}, Breuer, Guralnick and Kantor make the following remarkable conjecture.

\begin{conjecture*}
A finite group is $\frac{3}{2}$-generated if and only if every proper quotient is cyclic.
\end{conjecture*}

This conjecture has recently been reduced by Guralnick \cite{ref:GuralnickPC} to almost simple groups $G$. By the main theorem of \cite{ref:DallaVoltaLucchini95}, these groups are 3-generated and, in fact, 2-generated if $G/\soc(G)$ is cyclic, where $\soc(G)$ denotes the (simple) socle of $G$. In the case where $\soc(G)$ is alternating the conjecture was established in \cite{ref:Binder68}, and the sporadic groups are handled in \cite{ref:BreuerGuralnickKantor08} using computational methods. Therefore, it remains to consider the almost simple groups of Lie type. In \cite{ref:BurnessGuest13}, Burness and Guest establish a stronger version of the conjecture for almost simple linear groups. The aim of this paper is to extend this result to almost simple symplectic and odd-dimensional orthogonal groups. We will handle the remaining groups of Lie type in a forthcoming paper. \par

Following Brenner and Wiegold \cite{ref:BrennerWiegold75}, a finite group $G$ has \emph{spread} $k$ if for any $k$ non-identity elements $x_1,\dots,x_k \in G$ there exists $g \in G$ such that, for all $i$, $\<x_i,g\>=G$. Moreover, $G$ is said to have \emph{uniform spread} $k$ if the element $g$ can be chosen from a fixed conjugacy class of $G$. We write $s(G)$ (respectively $u(G)$) for the greatest $k$ such that $G$ has spread $k$ (respectively uniform spread $k$). (If $G$ is cyclic, then write $s(G) = u(G) = \infty$.) \par

In \cite{ref:BreuerGuralnickKantor08}, using probabilistic methods, it was proved that $u(G) \geq 2$ for all finite simple groups $G$, with equality if and only if \[ G \in \{ A_5, A_6, \Om^{+}_8(2)\} \cup \{ \Sp_{2m}(2) \mid m \geq 3 \}.\] This was extended by Burness and Guest in \cite{ref:BurnessGuest13}, where they prove that $u(G) \geq 2$ for $G=\<\PSL_n(q),g\>$ with $g \in \Aut(\PSL_n(q))$ unless $G=\PSL_2(9).2 \cong S_6$, for which $u(G)=0$ and $s(G)=2$. In particular, this demonstrates that $\< \PSL_n(q), g \>$ is $\frac{3}{2}$-generated. \par

Let us now introduce the groups which will be the focus of this paper. Write 
\begin{align}
\T = \{ \PSp_{2m}(q)' &\mid m \geq 2 \} \cup \{ \Om_{2m+1}(q) \mid q \text{ odd}, m \geq 3 \} \label{eq:DefT}  \\[4pt] 
\A = \{ &\< T, \th \> \mid T \in \T, \th \in \Aut(T)\} \label{eq:DefA}
\end{align} 
The restrictions on $m$ in the definition of $\T$ account for the familiar low-rank isomorphisms $\PSp_2(q) \cong \Om_3(q) \cong \PSL_2(q)$ and $\Om_5(q) \cong \PSp_4(q)$ (see \cite[Prop. 2.9.1]{ref:KleidmanLiebeck}).  \par

We can now present the main result of the paper.
\begin{mytheorem}\label{thm:MainResult}
Let $G \in \A$. Then $u(G) \geq 2$ unless $G = \PSp_4(2)'.2 \cong S_6$, in which case $u(G)=0$ and $s(G)=2$.
\end{mytheorem}

As an immediate consequence of Theorem~\ref{thm:MainResult}, all groups in $\A$ are $\frac{3}{2}$-generated. Therefore, this establishes the main conjecture for all almost simple groups whose socle is a symplectic group or odd-dimensional orthogonal group.

\begin{myremark}\label{rem:A6}
In the definition of $\T$, we take the derived subgroup of $\PSp_{2m}(q)$ since $\PSp_4(2) \cong S_6$ is not perfect. Accordingly, $A_6 \in \T$ and $\A$ includes $A_6$ together with the three cyclic extensions: $S_6$, $\PGL_2(9)$ and $M_{10}$. It is well-known that $u(A_6)=2$ and $u(S_6)=0$ but $s(S_6)=2$. Moreover, using \textsc{Magma} \cite{ref:Magma}, we can show that $u(\PGL_2(9))=5$ and $u(M_{10}) \geq 8$. (See Section~\ref{ssec:PrelimsComp} for a brief discussion of our computational methods.)
\end{myremark}

If we exclude some cases, we can strengthen the lower bound in Theorem~\ref{thm:MainResult}.
\begin{mytheorem}\label{thm:MainSharper}
Let $G \in \A$. Assume that $q$ is odd and $m \geq 3$. If $\soc(G) =  \Om_{2m+1}(q)$ then $u(G) \geq 3$, and if $\soc(G)=\PSp_{2m}(q)$ then $u(G) \geq 4$.
\end{mytheorem}

By \cite[Theorem 1.1]{ref:GuralnickShalev03}, if $(G_i)$ is a sequence of finite simple groups of Lie type such that $|G_i| \to \infty$, then $s(G_i) \to \infty$ if and only if $(G_i)$ does not have a subsequence of symplectic groups in even characteristic or odd-dimensional orthogonal groups, over a field of fixed size. We wish to establish similar results for sequences $(G_i)$ of almost simple groups of Lie type for which $G_i/\soc(G_i)$ is cyclic. (See \cite[Theorem 4]{ref:BurnessGuest13} for an asymptotic result for almost simple linear groups.) To this end, we prove the following result.
\begin{mytheorem}\label{thm:MainAsymptotic}
Let $(G_i)$ be a sequence of groups in $\A$ with $|G_i| \to \infty$. Then $u(G_i) \to \infty$ if and only if there is no subsequence $(G_{i_k})$ of groups over a field of fixed size such that either
\begin{enumerate}
\item{$\soc(G_{i_k})$ are symplectic groups in even characteristic; or}
\item{$\soc(G_{i_k})$ are odd-dimensional orthogonal groups.}
\end{enumerate}
\end{mytheorem}

We can find explicit bounds for the groups in Theorem~\ref{thm:MainAsymptotic} with bounded uniform spread.
\begin{mytheorem}\label{thm:MainUpper}
Let $G \in \A$. If $q$ is even, $\soc(G) = \PSp_{2m}(q)$ and $\th$ is not a graph-field automorphism, then $s(G) \leq q$.  If $\soc(G) = \Om_{2m+1}(q)$, then $s(G) < \frac{q^2+q}{2}$.
\end{mytheorem}

\begin{myremark}\label{rem:Bounded}
Let $q$ be even. Write $G=\< T, \th \>$ where $T = \PSp_4(q)'$ and $\th \in \Aut(T)$. 
\begin{enumerate}
\item{If $q=4$ and $\th$ is an involutory field automorphism, then it can be shown computationally that $u(G)=4$ (see Table~\ref{tab:Comp}). Therefore, the bound for symplectic groups in Theorem~\ref{thm:MainUpper} is sharp.}
\item{By \cite[Prop. 2.5]{ref:GuralnickShalev03}, $s(T) \leq q$. Theorem~\ref{thm:MainUpper} extends this result by establishing that if $\th$ is a field automorphism then $s(G) \leq q$. However, this upper bound does \emph{not} apply when $\th$ is a graph-field automorphism. Indeed, in this case, if $q=4$ then $u(G) \geq 10$, and, strikingly, if $q=8$ and $\th$ has order two then $u(G) \geq 76$. This behaviour is captured by Proposition~\ref{prop:ProbS4}(iii), which establishes that if $\th$ is an involutory graph-field automorphism then $u(G) \geq q^2/C$ for a constant $C$. (The proof of Proposition~\ref{prop:ProbS4}(iii) shows that we may choose $C=18$.) In particular, this gives infinitely many examples where $u(G) > u(\soc(G))$.}
\end{enumerate}
\end{myremark}

\begin{myremark}\label{rem:BoundedAsymp}
Let $q$ be even. Write $G=\< T, \th \>$ where $T = \PSp_{2m}(q)$ and $\th \in \Aut(T)$. By Proposition~\ref{prop:ProbSymp}(iv), if $m \geq 16$, then $q-1 \leq u(G) \leq s(G) \leq q$, so the upper bound for symplectic groups in Theorem~\ref{thm:MainUpper} is certainly close to best possible in large rank.
\end{myremark}

\begin{myremark}\label{rem:GeneratingGraph}
The above results can be recast combinatorially by way of the \emph{generating graph}. For a finite group $G$, let $\Ga(G)$ be the graph whose vertices are the non-identity elements of $G$ and in which two vertices $g$ and $h$ are adjacent if and only if $\<g,h\> = G$. This graph encodes many interesting generation properties of the group. For example, $\Ga(G)$ has no isolated vertices if and only if $G$ is $\frac{3}{2}$-generated. Further, if $s(G) \geq 2$, then $\Ga(G)$ is connected with diameter at most 2. Therefore, by \cite[Theorem 1.2]{ref:BreuerGuralnickKantor08}, the diameter of the generating graph of any non-abelian finite simple group is two. Moreover, Theorem~\ref{thm:MainResult} shows that the same conclusion holds for the groups in $\A$. \par

Many other natural questions about generating graphs have been investigated in recent years. For example, in \cite[Theorem 1.2]{ref:BreuerGuralnickLucchiniMarotiNagy10}, it is shown that for all sufficiently large simple groups $G$, the graph $\Ga(G)$ has a Hamiltonian cycle. Indeed, it is conjectured that for all finite groups $G$ of order at least four, the generating graph $\Ga(G)$ has a Hamiltonian cycle if and only if every proper quotient of $G$ is cyclic. This is a significant strengthening of the aforementioned conjecture of Breuer, Guralnick and Kantor, which asserts that the generating graph $\Ga(G)$ has no isolated vertices if and only if every proper quotient of $G$ is cyclic.  This stronger conjecture has been verified for soluble groups \cite[Prop. 1.1]{ref:BreuerGuralnickLucchiniMarotiNagy10}.
\end{myremark}

In the remainder of this introductory section, we will briefly discuss the main tools used in the proofs of Theorems~\ref{thm:MainResult}--\ref{thm:MainUpper}. As in \cite{ref:BurnessGuest13}, the main ingredient is the probabilistic method used by Guralnick and Kantor in \cite{ref:GuralnickKantor00}. Fix $G = \< T, \th \> \in \A$ and $s \in G$. Write $\M(G,s)$ for the set of maximal subgroups of $G$ which contain $s$. For $x \in G$, let $P(x,s)$ be the probability that $x$ and a random conjugate of $s$ do not generate $G$; that is, \[ P(x,s) = 1 - \frac{|\{z \in s^G \mid G = \< x,z \>\}|}{|s^G|}. \] By \cite[Lemma 2.1]{ref:BurnessGuest13}, $G$ has uniform spread $k$ if for all $k$-tuples $(x_1,\dots,x_k)$ of prime order elements in $G$, \[ \sum_{i=1}^{k}P(x_i,s) < 1. \]

To estimate $P(x,s)$ we use fixed point ratios. For a $G$-set $\Omega$, let $\fix(x,\Omega)$ be the number of fixed points of $x$ on $\Omega$ and let $\fpr(x,\Omega) = \fix(x,\Omega)/|\Omega|$ be the corresponding \emph{fixed point ratio}. For $x \in G$, by \cite[Lemma 2.2]{ref:BurnessGuest13}, 
\begin{equation}
P(x,s) \leq \sum_{H \in \M(G,s)}^{} \fpr(x,G/H). \label{eq:ProbMethod}
\end{equation}

Therefore, our probabilistic method has three steps: select an appropriate element $s \in G $, determine $\M(G,s)$ and use fixed point ratio estimates to bound $P(x,s)$ for each $x \in G$ of prime order. In the case where $\th$ is a field automorphism, we will use the theory of \emph{Shintani descent} to choose $s$ and control its maximal overgroups (see Section~\ref{ssec:PrelimsShintaniDescent}). \par

{\renewcommand{\arraystretch}{1.1}
\begin{tabularx}{0.95\textwidth}{ll}
\toprule[0.12em]
$\C_1$ & stabilisers of subspaces, or pairs of subspaces, of $V$ \\
$\C_2$ & stabilisers of decompositions $V = \bigoplus_{i=1}^{t}V_i$ where $\dim{V_i}=a$ \\
$\C_3$ & stabilisers of prime degree field extensions of $\F_q$ \\
$\C_4$ & stabilisers of tensor product decompositions $V=V_1 \otimes V_2$ \\
$\C_5$ & stabilisers of prime index subfields of $\F_q$ \\
$\C_6$ & normalisers of symplectic-type $r$-groups in absolutely irreducible representations \\
$\C_7$ & stabilisers of decompositions $V = \bigotimes_{i=1}^{t}V_i$ where $\dim{V_i}=a$ \\
$\C_8$ & stabilisers of non-degenerate forms on $V$ \\
\bottomrule[0.12em]
\caption{The collections of geometric subgroups} \label{tab:GeometricSubgroups}
\end{tabularx}}
\vspace{-5.5pt}

Our framework for understanding $\M(G,s)$ is provided by Aschbacher's subgroup structure theorem for finite classical groups \cite{ref:Aschbacher84}. Roughly, this theorem states that if $G$ is an almost simple classical group, then any maximal subgroup of $G$ not containing $\soc(G)$ belongs to one of eight collections $\C_1,\dots,\C_8$ of so-called geometric subgroups, or it is contained in $\S$, a collection of absolutely irreducible almost simple subgroups. The geometric subgroups preserve certain geometric structures on the natural module (see Table~\ref{tab:GeometricSubgroups}), and we refer the reader to \cite{ref:KleidmanLiebeck} for further details regarding these subgroups. A complete description of the maximal subgroups of classical groups of dimension at most 12 is given in \cite{ref:BrayHoltRoneyDougal}. For a maximal subgroup $H$ of $G$, the \emph{type} of $H$ is a rough indication of the structure of $H$. In addition to determining the types of subgroups in $\M(G,s)$, we need to calculate the multiplicity with which each type occurs. \par

Finally, in view of \eqref{eq:ProbMethod}, we use fixed point ratio estimates to bound $P(x,s)$. There is a vast literature on fixed point ratios for primitive actions of almost simple groups. If $G$ is a finite almost simple classical group, then the \emph{subspace subgroups} of $G$ are roughly the maximal subgroups which act reducibly on the natural module for $G$; that is, they are roughly the $\C_1$ subgroups. (For the precise definition see \cite[Definition 1]{ref:Burness071}.) In \cite{ref:Burness071,ref:Burness072,ref:Burness073,ref:Burness074}, Burness establishes close to best possible upper bounds on $\fpr(x,G/H)$ when $G$ is an almost simple classical group, $H$ is a maximal non-subspace subgroup and $x \in G$ has prime order. In particular, if $n$ is the dimension of the natural module for $G$, then \[ \fpr(x,G/H) \leq |x^G|^{-\frac{1}{2}+o(1)}, \] where $o(1) \to 0$ as $n \to \infty$. An explicit exponent is given in \cite[Theorem 1]{ref:Burness071}. For subspace subgroups we will use the bounds of Guralnick and Kantor in \cite[\SS 3]{ref:GuralnickKantor00}, together with some new bounds we establish in Section~\ref{sec:FPRs}. \par

For some low-dimensional groups over small fields, our probabilistic approach is complemented by computational methods implemented in \textsc{Magma} \cite{ref:Magma}. We refer the reader to Section~\ref{ssec:PrelimsComp} for the details.

\subsection*{Acknowledgements}
The author would like to thank his PhD supervisor Dr Tim Burness for bringing this problem to his attention, and he acknowledges the financial support of EPSRC and the Heilbronn Institute for Mathematical Research. He also thanks Prof Robert Guralnick for helpful advice.


\section{Preliminaries}\label{sec:Prelims}
In this section, we record preliminary results and fix notation.

\subsection{Symplectic and orthogonal groups}\label{ssec:PrelimsGroups}
Let us begin by discussing the almost simple groups which will be the focus of this paper. Let $q=p^f$ where $p$ is prime and let $V = \F_q^n$. For finite classical groups we will use the notation and terminology adopted by Kleidman and Liebeck in \cite{ref:KleidmanLiebeck} and Burness and Giudici in \cite{ref:BurnessGiudici16}. \par

{\renewcommand{\arraystretch}{1.1}
\begin{tabularx}{0.95\textwidth}{cccCc}
\toprule[0.12em]
Case            & $T$             & $n$            & Forms on $V=\F_q^n$                                               & Conditions on $q$ \\
\midrule
$\mathbf{S}$    & $\PSp_{2m}(q)$  & $2m \geq 6$    & symplectic form $( \ , \ )$                                       & none               \\
$\mathbf{S_4}$  & $\PSp_4(q)$     & $4$            & symplectic form $( \ , \ )$                                       & $q > 2$            \\
$\mathbf{O}$    & $\Om_{2m+1}(q)$ & $2m+1 \geq 7$  & non-degenerate quadratic form $Q$ with symmetric form $( \ , \ )$ & $q$ odd            \\
\bottomrule[0.12em]
\caption{The three cases for groups in $\A$} \label{tab:Groups}
\end{tabularx}}

Let $( \ , \ )$ be a bilinear form on $V$. The corresponding similarity group $\Delta(V)$ is the subgroup of $\GL(V)$ containing the elements $g$ for which there exists $\t(g) \in \F_{q}$ such that $(ug,vg) = \t(g)(u,v)$ for all $u,v \in V$. We refer to $\t\:\Delta(V) \to \F_{q}^{\times}$ as the \emph{similarity map}.  \par

Write $\Sp(V), \GSp(V), \mathrm{\Gamma \Sp}(V)$ for the groups of isometries, similarities and semisimilarities of $V$ with respect to a symplectic (i.e. non-degenerate alternating) form, and respectively $\O(V), \GO(V), \mathrm{\Gamma O}(V)$ for an odd-dimensional space $V$ with a non-degenerate symmetric form (over a field of odd characteristic). Let $\SO(V)$ be the index two subgroup of $\O(V)$ of maps with determinant one. The kernel of the spinor norm $\eta \: \SO(V) \to \F_q^{\times}/(\F_q^{\times})^2$ (see \cite[pp. 29--30]{ref:KleidmanLiebeck}) is the unique index two subgroup $\Om(V)$ of $\SO(V)$. \par

The sets $\T$ and $\A$ were introduced in \eqref{eq:DefT} and \eqref{eq:DefA}. In Table~\ref{tab:Groups}, we partition $\A$ into three subsets. (We omit groups with socle $\PSp_4(2)' \cong A_6$; see Remark~\ref{rem:A6}.) In each case, we define a formed space $V=\F_q^n$, which is the natural module for $T$.

Let $T \in \T$. We will now determine the possible groups $\< T, \th \>$ where $\th \in \Aut(T)$. To do this, it suffices to consider, for the choice of $\th$, the representatives of the outer automorphisms of $T$. By \cite[Theorem 30]{ref:Steinberg67}, $\Out(T)$ is generated by \emph{diagonal}, \emph{field}, \emph{graph} and \emph{graph-field automorphisms}. (We adopt the terminology of \cite[Definition 2.5.13]{ref:GorensteinLyonsSolomon98}.) The structure of $\Out(T)$ is easily determined, and we can identify the possibilities for $\th$ (see Table~\ref{tab:Theta}).

Let us define the notation used in Table~\ref{tab:Theta}. For $f > 1$, let $\p \in \Aut(T)$ be the field automorphism of order $f$ defined as $\overline{(a_{ij})} \mapsto \overline{(a_{ij}^p)}$, for each $\overline{(a_{ij})} \in T$. (Here we write linear maps on $V$ as matrices with respect to a standard basis for $V$ (see \cite[Prop 2.5.3]{ref:KleidmanLiebeck}), and we use overlines to denote reduction modulo scalars.)  Moreover, in case $\mathbf{S_4}$, if $q$ is even, let $\r$ be a graph-field automorphism of order $2f$ such that $\r^2 = \p$ (see \cite[Prop. 12.3.3]{ref:Carter72}). Finally, in cases $\mathbf{S}$ and $\mathbf{S_4}$ (respectively case $\mathbf{O}$), if $q$ is odd, let $\d$ be a diagonal automorphism of order 2 induced by an element of $\GSp_{2m}(q) \setminus \Sp_{2m}(q)$ (respectively $\SO_{2m+1}(q) \setminus \Om_{2m+1}(q)$). We write $\InnDiag(T)$ for the subgroup of $\Aut(T)$ generated by inner and diagonal automorphisms.

Now consider the elements of $G$. The conjugacy classes of elements of prime order in $G$ are described in \cite[\SS 3.4--3.5]{ref:BurnessGiudici16}, and we will refer to the relevant results when they are required. By \cite[Lemmas 3.4.2, 3.5.3]{ref:BurnessGiudici16}, the conjugacy class of an odd order semisimple element $g$ of $\GSp_n(q)$ or $\SO_n(q)$ is determined by the eigenvalues of $g$ over $\K$. Therefore, up to conjugacy, we will write $[\l_1,\dots,\l_n]$ for $g$, where $\l_1, \dots, \l_n \in \K$ are the eigenvalues of $g$. 

{\renewcommand{\arraystretch}{1.1}
\begin{tabularx}{0.8\textwidth}{ccccc}
\toprule[0.12em]
Case                                         & $q$  & $\Aut(T)$         & $\Out(T)$        & $\th$                 \\
\midrule
$\mathbf{S}$                                 & even & $\< T, \p \>$     & $C_f$            & $1, \p^i$             \\
$\mathbf{S_4}$                               & even & $\< T, \r \>$     & $C_{2f}$         & $1, \r^j, \p^i$       \\
$\mathbf{S}$, $\mathbf{S_4}$, $\mathbf{O}$   & odd  & $\< T, \d, \p \>$ & $C_2 \times C_f$ & $1, \d, \p^i, \d\p^i$ \\
\bottomrule[0.12em]
\caption{The possibilities for $\th$ \\[2pt] ($1 \leq i < f$ and $1 \leq j < 2f$ with $j$ odd)} \label{tab:Theta}
\end{tabularx}}

\subsection{Shintani descent} \label{ssec:PrelimsShintaniDescent}
Let $G = \<T,\th \> \in \A$ (see \eqref{eq:DefT} and \eqref{eq:DefA}). The first step of our probabilistic method is to select a $G$-class $s^G$, with respect to which we will analyse the uniform spread of $G$. It is straightforward to see that we must choose $s \in G \setminus T$, so we will choose $s \in T\th$. We need to control the maximal subgroups of $G$ which contain $s$, and the technique of Shintani descent from the theory of algebraic groups will allow us to do this. \par

Following \cite[\SS 2.6]{ref:BurnessGuest13}, let $X$ be a connected linear algebraic group over an algebraically closed field and let $\s\: X \to X$ be a Steinberg morphism. Write $X_{\s}$ for the (necessarily finite) fixed point subgroup of $X$ under $\s$. Let $e>1$ and observe that $X_{\s^e}$ is $\s$-stable. Therefore, $\s$ restricts to an automorphism of $X_{\s^e}$ and, with a slight abuse of notation, we may consider the semidirect product $G_1=X_{\s^e}\sp\<\s\>$.

\begin{remark}\label{rem:SigmaClarification}
Let us clarify our use of the symbol $\s$. Fix $g \in X_{\s^e}$. In one sense, $\s$ is a Steinberg morphism of $X$ which restricts to an automorphism of $X_{\s^e}$. Therefore, $\s(g)$ denotes the image of $g$ under the map $\s$. In a second sense, $\s$ is an element of the semidirect product $G_1=X_{\s^e}\sp\<\s\>$, so by $g\s$ we mean the product of $g$ and $\s$ in $G_1$ and by $g^{\s}$ we mean $\s^{-1}g\s$. By the definition of the semidirect product $G_1$, $g^{\s} = \s(g)$, so $g^{\s}$ will be our preferred way of referring to $\s(g)$. 
\end{remark}

Let $g \in X_{\s^e}$. By the Lang-Steinberg Theorem \cite[Theorem 21.7]{ref:MalleTesterman11}, there exists $a \in X$ such that $g=aa^{-\s^{-1}}$. Define the \emph{Shintani map} as \[f\: \{(g\s)^{G_1} \mid g \in X_{\s^e} \} \to \{x^{X_{\s}} \mid x \in X_{\s} \} \quad g\s \mapsto a^{-1}(g\s)^ea, \] for $a \in X$ such that $g=aa^{-\s^{-1}}$. We abuse notation by writing $f(g\s)$ for a representative of the class given by $f(g\s)$. The following combines \cite[Lemma 2.13, Theorem 2.14]{ref:BurnessGuest13}.
\begin{theorem}\label{thm:ShintaniDescent}
Let $X$ be the algebraic group, $\s$ be the Steinberg morphism and $f$ be the Shintani map as above.
\begin{enumerate}
\item{The Shintani map $f$ is a well-defined bijection.}
\item{For all $g \in X_{\s^e}$, $C_{X_{\s^e}}(g\s) \cong aC_{X_{\s}}(f(g\s))a^{-1}$.}
\item{Let $Y$ be a closed connected $\s$-stable subgroup of $X$. Then for all $g \in X_{\s^e}$, \[ \fix(g\s, X_{\s^e}/Y_{\s^e}) = \fix(f(g\s), X_{\s}/Y_{\s}).\]}
\end{enumerate}
\end{theorem}

\begin{remark}\label{rem:ShintaniDescent}\new
\begin{enumerate}
\item{In \cite[\SS 2.6]{ref:BurnessGuest13}, it is verified that $(g\s)^{G_1} = (g\s)^{X_{\s^e}}$, for all $g \in X_{\s^e}$. Consequently, $f$ is a bijection from $\{(g\s)^{X_{\s^e}} \mid g \in X_{\s^e} \}$ to $\{x^{X_{\s}} \mid x \in X_{\s} \}$.}
\item{The hypothesis that $\s$ is a Steinberg morphism is used (via the Lang-Steinberg Theorem) to guarantee the existence of the element $a \in X$ required to define the map $f$. However, the rest of the proof of Theorem~\ref{thm:ShintaniDescent} holds whenever $\s$ is an automorphism of $X$ as an abstract group with a finite fixed point subgroup. Therefore, we may define a Shintani map for any abstract automorphism $\s$ of $X$ which has a finite fixed point subgroup and such that for all $g \in X_{\s^e}$ there exists $a \in X$ for which $aa^{-\s^{-1}} = g$.}
\end{enumerate}
\end{remark}

We will now demonstrate how we will apply this general theory to our specific settings. Let $T \in \T$ and $\th \in \Aut(T)$. Assume that $\th \not\in \InnDiag(T)$. Let $K = \K$ and define
\begin{equation}
X = \left\{ 
\begin{array}{ll}
\PGSp_{2m}(K)  & \text{if $T= \PSp_{2m}(q)$} \\
\SO_{2m+1}(K) & \text{if $T = \Om_{2m+1}(q)$} \\
\end{array}
\right. \label{eq:DefX}
\end{equation}
Abusing notation, as explained in Remark~\ref{rem:SigmaClarification}, let $\p\:X \to X$ be defined as $\overline{(a_{ij})} \mapsto \overline{(a_{ij}^p)}$. Observe that $\p|_T$ is the automorphism $\p$ from Table~\ref{tab:Theta}. We will split into two cases depending on whether $\th$ is a graph-field automorphism.

\subsubsection{Automorphisms other than graph-field automorphisms}\label{sssec:PrelimsShintaniField}
Assume that $\th$ is not a graph-field automorphism. Then, from Table~\ref{tab:Theta}, $\th = \p^i$ or $\th = \d\p^i$ for some $1 \leq i < f$. Let $e$ be the order of $\p^i$ and write $q=q_0^e$. Then define the Frobenius morphism $\s\: X \to X$ as
\begin{equation}
\s = \p^i. \label{eq:DefSigmaField}
\end{equation}
That is, $\s$ is defined as $\overline{(a_{ij})} \mapsto \overline{(a_{ij}^{q_0})}$. Therefore, if $T=\Sp_{2m}(q)$ and $q$ is even, then $\th = \p^i$, and we obtain the Shintani map
\begin{equation}
f\: \{(t\th)^{\Sp_{2m}(q)} \mid t \in \Sp_{2m}(q) \} \to \{x^{\Sp_{2m}(q_0)} \mid x \in \Sp_{2m}(q_0) \}. \label{eq:SympShintaniEvenField}
\end{equation} 
Additionally, if $q$ is odd (and $T=\PSp_{2m}(q)$ or $T=\Om_{2m+1}(q)$), then we obtain the maps
\begin{align}
& f\: \{(g\p^i)^{\PGSp_{2m}(q)} \mid g \in \PGSp_{2m}(q) \} \to \{x^{\PGSp_{2m}(q_0)} \mid x \in \PGSp_{2m}(q_0) \}; \label{eq:SympShintaniOdd} \\
& f\: \{(g\p^i)^{\SO_{2m+1}(q)} \mid g \in \SO_{2m+1}(q) \} \to \{x^{\SO_{2m+1}(q_0)} \mid x \in \SO_{2m+1}(q_0) \}. \label{eq:OrthShintani}
\end{align} 

However, for our application, we need to study cosets of $T$ rather than of $\InnDiag(T)$. The following propositions allow us to do this.
\begin{proposition}\label{prop:SympShintaniOdd}
Let $q$ be odd and let $T=\PSp_{2m}(q)$. The Shintani map $f$ restricts to bijections
\begin{enumerate}
\item{$f_1\: \{(t\p^i)^{\PGSp_{2m}(q)}   \mid t \in T \} \to \{x^{\PGSp_{2m}(q_0)} \mid x \in \PSp_{2m}(q_0) \};$}
\item{$f_2\: \{(t\d\p^i)^{\PGSp_{2m}(q)} \mid t \in T \} \to \{x^{\PGSp_{2m}(q_0)} \mid x \in \PGSp_{2m}(q_0) \setminus \PSp_{2m}(q_0) \}.$}
\end{enumerate} 
\end{proposition}

\begin{proposition}\label{prop:OrthShintani}
Let $T=\Om_{2m+1}(q)$. The Shintani map $f$ restricts to bijections
\begin{enumerate}
\item{$f_1\: \{(t\p^i)^{\SO_{2m+1}(q)}   \mid t \in T \} \to \{x^{\SO_{2m+1}(q_0)} \mid x \in \Om_{2m+1}(q_0) \};$}
\item{$f_2\: \{(t\d\p^i)^{\SO_{2m+1}(q)} \mid t \in T \} \to \{x^{\SO_{2m+1}(q_0)} \mid x \in \SO_{2m+1}(q_0) \setminus \Om_{2m+1}(q_0) \}.$}
\end{enumerate}
\end{proposition}

We will prove Proposition~\ref{prop:SympShintaniOdd}; the proof of Proposition~\ref{prop:OrthShintani} is analogous, replacing the similarity map $\t$ with the spinor norm $\eta$. \par

For all $k$, there are natural embeddings $\PSp_{2m}(p^k) \hookrightarrow \PGSp_{2m}(p^k) \hookrightarrow \PSp_{2m}(K)$. Thus, we will identify each of the symplectic groups in the following proof with suitable subgroups of $\PSp_{2m}(K)$ and write $Z=Z(\Sp_{2m}(K))$. Let $N\:\F_{q} \to \F_{q_0}$ be the norm map. 
\begin{proof}[Proof of Proposition~\ref{prop:SympShintaniOdd}] 
Fix $g \in \PGSp_{2m}(q)$ and let $y \in \PGSp_{2m}(q_0)$ be a representative of the conjugacy class $f(g\s)$. Write $\s = \p^i$ and let $a \in \PSp_{2m}(K)$ be such that $g=aa^{-\s^{-1}}$. We will now take suitable lifts of elements. Write $a=\hat{a}Z$ and $\hat{g}=\hat{a}{\hat{a}}^{-\s^{-1}} \in \GSp_{2m}(q)$. So $g=\hat{g}Z$. Therefore, $y = f(g\s) = \hat{a}^{-1}(\hat{g}\s)^e\hat{a}Z$, and, accordingly, write $\hat{y} = \hat{a}^{-1}(\hat{g}\s)^e\hat{a}$. \par

Now we wish to connect $\t(\hat{y})$ and $\t(\hat{g})$. Observe that \[ \t(\hat{y}) = \t(\hat{a}^{-1}(\hat{g}\s)^e\hat{a}) =\t((\hat{g}\s)^e) = \t(\hat{g})\t(\hat{g}^{\s^{e-1}})\cdots\t(\hat{g}^{\s}) =  N(\t(\hat{g})), \] since $\t(\hat{g}^{\s^k}) = \t(\hat{g})^{\s^k}$ for all $k$. In particular, $\t(\hat{g})$ is a square in $\F_{q}$ if and only if $\t(\hat{y})$ is a square in $\F_{q_0}$. That is, $g \in \PSp_{2m}(q)$ if and only if $f(g\s) = y \in \PSp_{2m}(q_0)$. Therefore, restricting $f$ to $\PGSp_{2m}(q)$-classes of $T\s$ and $T\d\s$ gives the required bijections.
\end{proof}

\subsubsection{Graph-field automorphisms}\label{sssec:PrelimsShintaniGraphField}
Now assume that $q$ is even and $T=\Sp_4(q)$. Let $\th$ be the graph-field automorphism $\r^j$ where $j$ is an odd integer such that $1 \leq j < 2f$ (see Table~\ref{tab:Theta}). Then $(\r^j)^2 = \p^j$ since, by definition, $\r^2 = \p$. Let $e$ be the order of $\p^j$ and write $q=q_0^e$. Therefore, $q_0=2^j$. Abusing notation as above, let $\r\:X \to X$ be a Steinberg morphism such that $\r^2=\p$. Define the Steinberg morphism $\s\: X \to X$ as
\begin{equation}
\s = \r^j. \label{eq:DefSigmaGraphField}
\end{equation}

The following proposition describes the Shintani map given by the above setup. (In this result, $Sz(q_0)$ is the Suzuki group over the field $\F_{q_0}$.)
\begin{proposition}\label{prop:SympShintaniEvenGraphField}
Let $T=\Sp_4(q)$ with $q > 2$ even and let $\th = \r^j$. Then there is a Shintani map \[ f \: \{ (t\th)^T \mid t \in T \} \to \{ x^{Sz(q_0)} \mid x \in Sz(q_0) \} \quad t\th \mapsto a^{-1}(t\th)^{2e}a. \]
\end{proposition} 

\begin{proof}
By Theorem~\ref{thm:ShintaniDescent}, there is a Shintani map $f$ from the $X_{(\r^j)^{2e}}$-classes in $X_{(\r^j)^{2e}}\r^j$ to the $X_{\r^j}$-classes in $X_{\r^j}$. Since $2^{je}=q_0^e=q$, $X_{(\r^j)^{2e}} = \Sp_{4}(q) = T$ and the restriction of $\r^j$ to $X_{(\r^j)^{2e}}$ is the automorphism $\th$. Similarly, $X_{\r^j} = C_{X_{\r^{2j}}}(\r^j) = C_{\Sp_4(q_0)}(\r^j)$. Since $\r^j$ is an involutory graph-field automorphism of $\Sp_4(q_0)$, by \cite[(19.4)]{ref:AschbacherSeitz76}, $C_{\Sp_4(q_0)}(\r^j) \cong Sz(q_0)$. This proves the result.
\end{proof}

\subsubsection{Applications}\label{sssec:PrelimsShintaniApplications}
Let us now record several applications of Shintani descent to the problem of studying the maximal overgroups of particular elements, a crucial component of our probabilistic approach. For the remainder of this section, let $G = \<T,\th\> \in \A$ and recall the formed space $V=\F_q^n$ from Table~\ref{tab:Groups}. Moreover, let $X$ be the algebraic group defined in \eqref{eq:DefX}, let $\s$ be the Steinberg morphism defined in $\eqref{eq:DefSigmaField}$ or $\eqref{eq:DefSigmaGraphField}$ and let $f$ be the Shintani map defined in \eqref{eq:SympShintaniEvenField}--\eqref{eq:OrthShintani} or Proposition~\ref{prop:SympShintaniEvenGraphField}. Recall that we write $G_1=X_{\s^e}\sp\<\s\>$. \par

Our first result, which is \cite[Prop. 2.16(i)]{ref:BurnessGuest13}, provides a general bound.
\begin{proposition} \label{prop:CentraliserBound}
Let $H$ be a maximal subgroup of $G$ and let $g\s \in G$. Then $g\s$ is contained in at most $|C_{X_{\s}}(f(g\s))|$ $G_1$-conjugates of $H$. 
\end{proposition}

Proposition~\ref{prop:CentraliserBound} is notable for both its effectiveness and its generality. However, for some particular subgroups we require a tighter bound for our probabilistic estimates. For instance, the following result is modelled on \cite[Corollary 2.15]{ref:BurnessGuest13} and the proof is similar.
\begin{proposition} \label{prop:ShintaniTransfer}
Let $Y$ be the stabiliser in $X$ of a totally isotropic $k$-space, or, in the case where $T=\PSp_{2m}(q)$, the stabiliser of a non-degenerate $k$-space with $k < m$. Assume that $Y$ is $\s$-stable. For all $g \in X_{\s^e}$, the number of $X_{\s^e}$-conjugates of $Y_{\s^e}$ which are normalised by $g\s$ is equal to the number of $X_{\s}$-conjugates of $Y_{\s}$ which contain $f(g\s)$.
\end{proposition}

In even characteristic, we can also use Shintani descent to determine the number of orthogonal subgroups of a symplectic group which contain a given element. Let $q$ be even and recall that $K = \K$. Let $X=\Sp_{2m}(K)$ and $Y = \O_{2m+1}(K)$, the isometry group of a non-singular quadratic form on $K^{2m+1}$ (see \cite[pp. 143--144]{ref:Taylor92}). By \cite[Theorem 11.9]{ref:Taylor92}, there exists an isomorphism $\psi\:X \to Y$ of abstract groups. \par

Let $\s\:X \to X$ be a Frobenius morphism, and define $\tau\:Y \to Y$ as $\tau = \psi \circ \s \circ \psi^{-1}$. It is straightforward to verify that $\psi$ extends to an isomorphism $\psi\:X\sp\<\s\> \to Y\sp\<\t\>$ by defining $\psi(\s) = \t$. By Theorem~\ref{thm:ShintaniDescent}, for $e > 1$, we have a Shintani map \[f\:\{(g\s)^{X_{\s^e}} \mid g \in X_{\s^e} \} \to \{x^{X_{\s}} \mid x \in X_{\s}\}.\]  Since $Y_{\t^e}$ is $\t$-stable and $\psi$ restricts to an isomorphism $\psi\:X_{\s^e}\sp\<\s\> \to Y_{\t^e}\sp\<\t\>$, define \[f'\:\{(h\t)^{Y_{\t^e}} \mid h \in Y_{\t^e} \} \to \{y^{Y_{\t}} \mid y \in Y_{\t}\}\] as $f'=\psi \circ f \circ \psi^{-1}$. (Recall the notation for automorphisms explained in Remark~\ref{rem:SigmaClarification}.)

\begin{lemma}\label{lem:PseudoLangSteinberg}
With the notation above, for all $h \in Y_{\t^e}$ there exists $b \in Y$ such that $bb^{-\t^{-1}}=h$ and $f'(h\t) = b^{-1}(h\t)^eb$.
\end{lemma}

\begin{proof}
Let $g \in X_{\s^e}$ such that $\psi(g)=h$, and let $a \in X$ such that $aa^{-\s^{-1}}=g$. Then \[ f'(h\t) = \psi(f(\psi^{-1}(h\t))) = \psi(f(g\s)) = \psi(a^{-1}(g\s)^ea) = \psi(a)^{-1}(h\t)^e\psi(a) \] where \[ \psi(a)\psi(a)^{-\t^{-1}} = \psi(a)\t^{-1}(\psi(a^{-1})) = \psi(a)\psi(\s^{-1}(a^{-1})) = \psi(aa^{-\s^{-1}}) = \psi(g) = h. \qedhere \]
\end{proof}

Although $\tau$ need not be a Frobenius morphism of $Y$, by Remark~\ref{rem:ShintaniDescent}(ii), the conclusion of Theorem~\ref{thm:ShintaniDescent} holds for $f'$ as Lemma~\ref{lem:PseudoLangSteinberg} guarantees that the required $b \in Y$ exists. \par

Let $\s$ be the standard Frobenius morphism with fixed field $\F_{q_0}$ and write $q=q_0^e$. Then $X_{\s^e} = \Sp_{2m}(q)$ and $X_{\s} = \Sp_{2m}(q_0)$. The author thanks Prof Robert Guralnick for helpful comments on the proof of the following proposition.
\begin{proposition}\label{prop:MultiplicityOrthogonal}
With the notation above, for all $g\s \in \Sp_{2m}(q) \sp \<\s\>$ the total number of maximal subgroups of $\Sp_{2m}(q)\sp\<\s\>$ of type $\O^+_{2m}(q)$ or $\O^-_{2m}(q)$ which contain $g\s$ equals the total number of subgroups of $\Sp_{2m}(q_0)$ of type $\O^+_{2m}(q_0)$ or $\O^-_{2m}(q_0)$ which contain $f(g\s)$.
\end{proposition}

\begin{proof}
The maximal subgroups of $\Sp_{2m}(q)\sp\<\s\>$ of type $\O^{\pm}_{2m}(q)$ which contain $g\s$ correspond to the maximal subgroups of $\O_{2m+1}(q)$ of type $\O^{\pm}_{2m}(q)$ which are normalised by $\psi(g\s)$, and these are exactly the stabilisers of non-degenerate hyperplanes of $W=\F_q^{2m+1}$. \par

If a hyperplane $U$ is non-degenerate, then $U$ does not contain the radical $W \cap W^{\perp} = \< v \>$. We claim that the converse also holds. To see this, assume $v \not\in U$ and suppose that $x \in U \cap U^{\perp}$ is non-zero. Since $v \not\in U$, we know that $x \not\in \rad(W)$. Hence, there exists $w \in W$ such that $(x,w) \neq 0$. Therefore, $w \not\in U$ and, hence, $W = \< U, w \>$. In particular, $v = u + \l w$ for some $u \in U$ and $\l \neq 0$. Then $(x,v) = (x,u) + \l (x,w) = 0 + \l (x,w) \neq 0$ since $\l \neq 0$ and $(x,w) \neq 0$. However, $(x,v)=0$ since $v \in W \cap W^{\perp}$, which is a contradiction. Therefore, $U \cap U^{\perp} = 0$, so $U$ is non-degenerate. To summarise, the maximal subgroups of $\O_{2m+1}(q)$ of type $\O^{\pm}_{2m}(q)$ are exactly the stabilisers of hyperplanes not containing $v$. \par

Therefore, the maximal subgroups of $\Sp_{2m}(q)\sp\<\s\>$ of type $\O^{\pm}_{2m}(q)$ which contain $g\s$ correspond to the stabilisers in $\O_{2m+1}(q)$ of hyperplanes not containing $v$ which are normalised by $\psi(g\s)$. By lifting to $\SL_{2m+1}(q)$ and applying \cite[Corollary~2.15]{ref:BurnessGuest13} (a consequence of Theorem~\ref{thm:ShintaniDescent}(iii) which is the analogue of Proposition~\ref{prop:ShintaniTransfer} in the linear case), the stabilisers in $\SL_{2m+1}(q)$ of hyperplanes not containing $v$ which are normalised by $\psi(g\s)$ correspond to the stabilisers in $\SL_{2m+1}(q_0)$ of hyperplanes not containing the radical of $\F_{q_0}^{2m+1}$ which contain $f'(\psi(g\s))$. By the argument of the previous paragraph, the intersections of these subgroups with $\O_{2m+1}(q_0)$ are exactly the maximal subgroups of $\O_{2m+1}(q_0)$ of type $\O^{\pm}_{2m}(q_0)$ which contain $f'(\psi(g\s))$. These subgroups correspond to the maximal subgroups of $\Sp_{2m}(q_0)$ of type $\O^{\pm}_{2m}(q_0)$ which contain $\psi^{-1}(f'(\psi(g\s))) = f(g\s)$. 
\end{proof}

Let us record a consequence of Proposition~\ref{prop:MultiplicityOrthogonal}.
\begin{corollary}\label{cor:OrthogonalSubgroups}
Let $q$ be even and let $G = \Sp_{2m}(q)\sp\< \phi \>$, where $\phi$ is a field automorphism of $\Sp_{2m}(q)$. Then every element of $G$ is contained in at least one maximal subgroup of type $\O_{2m}^+(q)$ or $\O_{2m}^-(q)$.
\end{corollary}

\begin{proof}
If $x \in \Sp_{2m}(q)$, then, by \cite[Theorem 2]{ref:Dye79}, $x$ is contained in at least one subgroup $H$ of type $\O_{2m}^{\pm}(q)$. Hence, $x \in N_G(H)$, a maximal subgroup of $G$ of type $\O_{2m}^{\pm}(q)$. If $x \in G \setminus \Sp_{2m}(q)$, then $x=g\s$ where $\s$ is a power of $\phi$. Therefore, by Proposition~\ref{prop:MultiplicityOrthogonal}, the number of subgroups of $G$ of type $\O_{2m}^{\pm}(q)$ containing $g\s$ equals the number of subgroups of $\Sp_{2m}(q_0)$ of type $\O_{2m}^{\pm}(q_0)$ containing $f(g\s)$, which is at least one.
\end{proof}

\subsection{Computational methods} \label{ssec:PrelimsComp}
In addition to the probabilistic method described in the introduction, we carry out computations in \textsc{Magma} \cite{ref:Magma} to determine lower bounds on the uniform spread of some particular groups. In this way, for each group $G=\<T,\th\>$ in Table~\ref{tab:Comp}, we verify that $u(G) \geq k$. (Here we use the notation from Table~\ref{tab:Theta}.)

In each case, the group $G$ can be accessed directly, constructed using the command \texttt{AutomorphismGroup} or found as a subgroup of $\<\PSp_n(q), \p, \d \>$, which is obtained from $\mathrm{P\Sigma L}_n(q)$ by repeatedly using \texttt{MaximalSubgroups}. We have implemented an algorithm in \textsc{Magma} which takes as input a finite group \texttt{G}, positive integers $\texttt{k}$, $\texttt{N}$ and an element \texttt{s} in $G$ whose conjugacy class we wish to show witnesses the uniform spread $k$ of $G$. \par

{\renewcommand{\arraystretch}{1.1}
\begin{tabularx}{0.8\textwidth}{ccc}
\toprule[0.12em]
$T$           & $\th$           & $k$ \\
\midrule
$\PSp_4(3)$   & $\d$            & 2   \\
$\PSp_6(3)$   & $1, \d$         & 4   \\
$\Om_7(3)$    & $\d$            & 3   \\[5pt]
$\PSp_4(4)$   & $\p$            & 4   \\
$\PSp_4(4)$   & $\r$            & 10  \\
$\PSp_4(8)$   & $\p, \r$        & 2   \\
$\PSp_4(8)$   & $\r^3$          & 76  \\
$\PSp_4(16)$  & $\p, \p^2, \r$  & 2   \\[5pt]
$\PSp_4(9)$   & $\p, \d\p$      & 4   \\
$\PSp_4(25)$  & $\p, \d\p$      & 2   \\
$\PSp_4(27)$  & $\p, \d\p$      & 2   \\
\bottomrule[0.12em]
\caption{Computational results: $u(\<T,\th\>) \geq k$}\label{tab:Comp}
\end{tabularx}}

First, we follow the probabilistic method described in the introduction. To determine $\M(G,s)$ we use $\texttt{MaximalSubgroups}$. For each conjugacy class $x^G$, we need to compute $\fpr(x,G/H)$ for each $H \in \M(G,s)$; we do this by calculating $|x^G \cap H|$ using $\texttt{IsConjugate}$, noting that $\fpr(x,G/H) = \frac{|x^G \cap H|}{|x^G|}$. If for all $k$-tuples of classes $(x_1^G,\dots,x_k^G)$ we establish that $P(x_1,s) + \cdots + P(x_k,s) < 1$, then we have verified that $u(G) \geq k$ with respect to $s^G$. \par

Otherwise, for each $k$-tuple of classes $(C_1,\dots,C_k)$, we apply a randomised method (parameterised by $N$) to explicitly construct an element $z \in s^G$ such that for all $c_i \in C_i$, $\<c_1,z\> = \cdots = \<c_k,z\> = G$.  This randomised approach is based on the \textsf{GAP} calculations in \cite[\SS 4]{ref:BreuerGuralnickKantor08}, which are described by Breuer in \cite[\SS 3.3]{ref:Breuer07}. Observe that it suffices to show that for all representatives $(x_1,\dots,x_k)$ of the orbits of $C_1 \times \cdots \times C_k$ under the diagonal conjugation action of $G$, there exists $z \in s^G$ such that $\<x_1,z\> = \cdots = \<x_k,z\> = G$. An algorithm of \cite[pp. 18--19]{ref:Breuer07} to construct these orbit representatives is the crucial ingredient. Given these representatives, we test at most $N$ random conjugates of $s$ for each list of representatives, and we return any $k$-tuples of conjugacy classes for which no suitable conjugate of $s$ is found. If no $k$-tuples fail, then the bound $u(G) \geq k$ holds.


\section{Fixed point ratios}\label{sec:FPRs} 
This section serves to provide the fixed point ratio bounds we require in order to apply the probabilistic method in Section~\ref{sec:Proof}. There is a vast literature on fixed point ratios for primitive actions of almost simple groups. In addition to the essential role they play in random generation, these bounds have many other applications, such as to the study of base sizes (e.g. see \cite{ref:Burness07}) and monodromy groups (e.g.  see \cite{ref:FrohardtMagaard01}). \par

The most general bound in this area is \cite[Theorem 1]{ref:LiebeckSaxl91} of Liebeck and Saxl, which establishes that $\fpr(x,G/H) \leq 4/3q$, for any almost simple group of Lie type over $\F_q$, maximal subgroup $H \leq G$ and non-identity element $x \in G$, with a known list of exceptions. However, a theorem of Burness \cite[Theorem 1]{ref:Burness071} gives a stronger result when $G$ is a finite almost simple classical group, $H$ is a non-subspace subgroup and $x \in G$ has prime order. (Recall that, roughly, a subgroup of $G$ is \emph{non-subspace} if it acts irreducibly on the natural module for $G$; see \cite[Definition 1]{ref:Burness071}.) Namely, if $n$ is the dimension of the natural module for $G$, then \[ \fpr(x,G/H) \leq |x^G|^{-\frac{1}{2}+\frac{1}{n}+\i} \] where $\i$ is given in \cite[Table 1]{ref:Burness071}. \par

\vspace{5pt}

For remainder of this section, let $G$ be an almost simple group with socle $T \in \T$, where
\begin{align*}
\T = \{ \PSp_{2m}(q)' &\mid m \geq 2 \} \cup \{ \Om_{2m+1}(q) \mid q \text{ odd}, m \geq 3 \}. 
\end{align*} 
Assume that $T \neq \PSp_4(2)' \cong A_6$. \par

Let us introduce some notation.
\begin{notation}\label{not:nu}
Let $V=\F_q^n$ and let $x \in \PGL(V)$. Let $\hat{x}$ be a preimage of $x$ in $\GL(V)$. Define $\nu(x)$ to be the codimension of the largest eigenspace of $\hat{x}$ on $\olV = V \otimes_{\F_{q}} \K$.
\end{notation}

We begin by recording a consequence of \cite[Theorem 1]{ref:Burness071}. 
\begin{proposition}\label{prop:FPRs}
Let $x \in G$ have prime order and assume that $m \geq 3$. Suppose that $H$ is a maximal non-subspace subgroup of $G$. Let $\ell = 1$ unless specified otherwise in Table~\ref{tab:FPRsIota}.
\begin{enumerate}
\item{If $T=\Om_{2m+1}(q)$, then \[ \fpr(x,G/H) < \frac{(4q+4)^{1/2}}{q^{m-\ell+\e}}, \] where $\e=1/2$ unless $x \in \PGL(V)$ and $\nu(x)=1$, in which case $\e=0$.}
\item{If $T=\PSp_{2m}(q)$, then \[ \fpr(x,G/H) < \frac{(2q+2)^{1/2}}{q^{m-\ell}}. \]}
\item{If $T=\PSp_{2m}(q)$, then $\fpr(x,G/H) < F(x,G/H)$ as given in Table~\ref{tab:FPRsExtra}.}
\end{enumerate}
\end{proposition}

\begin{proof}
First suppose that $x \in \PGL(V)$. If $T = \PSp_{2m}(q)$ and $s=\nu(x)=1$, then $\max(s(2m-s),sm) = 2m-1$. Therefore, by \cite[Prop. 3.22, 3.36]{ref:Burness072},
\[|x^G| \geq |x^{\PSp_{2m}(q)}| > \frac{q^{2m}}{2q+2}.\] By \cite[Theorem 1]{ref:Burness071}, letting $\ell = 1 + 2m\i$, 
\begin{align*}
\fpr(x,G/H) &< \frac{1}{|x^G|^{1/2-1/{2m}-\i}} < \frac{(2q+2)^{1/2 - 1/2m - \i}}{q^{m-1-2m\i}} = \frac{(2q+2)^{1/2 - \ell/2m}}{q^{m-\ell}}.
\end{align*}
The remaining cases are similar. \par

Now assume that $x \not\in \PGL(V)$. Therefore, $x$ is a field automorphism, and
\[|x^G| \geq |x^T| \geq \frac{|\PSp_{2m}(q)|}{|\PGSp_{2m}(q^{\frac{1}{2}})|} > \frac{1}{2}q^{m^2+m/2} \] since $|\PSp_{2m}(q)|=|\Om_{2m+1}(q)|$. Then, by \cite[Theorem 1]{ref:Burness071},
\[ \fpr(x,G/H) < \frac{1}{|x^G|^{1/2-\ell/{2m}}} < \frac{2}{q^m}. \qedhere \]
\end{proof}

\begin{minipage}{0.4\textwidth}
{\renewcommand{\arraystretch}{1.2}
\begin{tabularx}{\textwidth}{ccc}
\toprule[0.12em]
$T$            & Type of $H$        & $\ell$ \\
\midrule  
$\PSp_{2m}(q)$ & $\Sp_m(q) \wr S_2$ & $2$    \\
$\PSp_{2m}(q)$ & $\Sp_m(q^2)$       & $2$    \\
$\Om_7(q)$     & $G_2(q)$           & $1.76$ \\
$\PSp_8(2)$    & $A_{10}$           & $1.50$ \\
$\Om_7(3)$     & $\Sp_6(2)$         & $1.46$ \\
$\PSp_6(2)$    & $\PSU_3(3)$         & $1.33$ \\
\bottomrule[0.12em]
\caption{Values of $\ell$}\label{tab:FPRsIota}
\end{tabularx}}
\end{minipage}
\begin{minipage}{0.63\textwidth}
{\renewcommand{\arraystretch}{1.2}
\begin{tabularx}{\textwidth}{cc}
\toprule[0.12em]
Condition on $x$                    & $F(x,G/H)$                                             \\
\midrule  
$x \in \PGL(V)$ and $\nu(x) = 1$    & $\dfrac{(2q+2)^{1/2-\ell/2m}}{q^{m-\ell}}$           \\[13pt]
$x \in \PGL(V)$ and $\nu(x) \geq 2$ & $\dfrac{(2q+2)^{1/2-\ell/2m}}{q^{2(m-\ell)-3/2+3/2m}}$ \\[13pt]
$x \not\in \PGL(V)$                 & $\dfrac{2}{q^m} $                                      \\
\bottomrule[0.12em]
\caption{Bounds for Proposition~\ref{prop:FPRs}(iii)}\label{tab:FPRsExtra}
\end{tabularx}}
\end{minipage}
\vspace{11pt}

In a special case of interest, we can provide a stronger result for subfield subgroups.
\begin{proposition}\label{prop:FPRsTrans}
Let $x \in G \cap \PGL(V)$ have prime order and assume that $\nu(x)=1$. Let $H$ be a maximal subfield subgroup of $G$. Then \[ \fpr(x,G/H) < 2q^{-m}.\]
\end{proposition}

\begin{proof}
By \cite[\SS 3.4--3.5]{ref:BurnessGiudici16}, a prime order element with $\nu(x) = 1$ is $G$-conjugate to the block diagonal matrix $[-I_{2m},1]$ in case $\mathbf{O}$, and $[J_2,I_{2m-2}]$ in cases $\mathbf{S}$ and $\mathbf{S_4}$. (Here we use $J_i$ to denote a Jordan block of size $i$.) Therefore, in each case, $x^G \cap H = x^H$.  The result follows from the centraliser orders in \cite[Appendix B]{ref:BurnessGiudici16}.
\end{proof}

Since \cite[Theorem 1]{ref:Burness071} excludes subspace subgroups, we use the bounds in \cite[\SS 3]{ref:GuralnickKantor00} in these cases. For convenience we record the relevant bounds below. (Recall that the \emph{Witt index} of an orthogonal space is the largest dimension of a totally singular subspace. Thus, the Witt index of a non-degenerate $(2m+1)$-space is $m$, and the Witt index of a non-degenerate $2m$-space is $m$ if the space is plus-type, and $m-1$ if the space is minus-type.)
\begin{proposition}\label{prop:FPRsSubspace}
Let $x \in G$ have prime order and assume that $m \geq 3$.
\begin{enumerate}[itemsep=5pt]
\item{Let $H$ be the stabiliser of a totally isotropic $k$-space, where $1 \leq k \leq m$. Then \[ \fpr(x,G/H) < 2q^{-(m-1)} + q^{-m} + q^{-k}.\]}
\item{Let $T=\Om_{2m+1}(q)$ and let $H$ be the stabiliser of a non-degenerate $k$-space of Witt index $l$, where $1 \leq k \leq 2m$. Then \[\fpr(x,G/H) < 2q^{-(m-1)} + q^{-m} + q^{-l} + q^{-(2m+1-k)}.\]}
\item{Let $T=\PSp_{2m}(q)$ and let $H$ be the stabiliser of a non-degenerate $k$-space, where $1 \leq k \leq 2m-1$. Then \[\fpr(x,G/H) < 2q^{-(m-\a)} + q^{-m} + q^{-k/2} + q^{-(2m-k)},\] where $\a = 1$ if $q$ is even, and $\a = 2$ if $q$ is odd.}
\item{Let $q$ be even, let $T=\Sp_{2m}(q)$ and let $H$ be a subgroup of $G$ of type $\O^{\pm}_{2m}(q)$. Then \[ \fpr(x,G/H) < q^{-\b} + q^{-m},\] where $\b=1$ if $x \in \PGL(V)$ and $\nu(x)=1$, and $\b=2$ otherwise.}
\end{enumerate} 
\end{proposition}

\begin{proof}
See \cite[Prop. 3.15, 3.16, Lemma 3.18]{ref:GuralnickKantor00}.
\end{proof}

Notice that the bounds in Proposition~\ref{prop:FPRsSubspace}(i)--(iii) do not depend on $x$. In contrast, \cite[Theorems 1--6]{ref:FrohardtMagaard00} provide upper and lower bounds for the fixed point ratio of an element $x$ of an almost simple classical group on an appropriate set of $k$-spaces which depend not only on $q$, $n$ and $k$, but also $\nu(x)$ when $x \in G \cap \PGL(V)$. However, for our application, the constants in these bounds are not sufficient. Therefore, we present bounds which are similar to those in \cite{ref:FrohardtMagaard00}, but with sharper constants in the special case that we are interested in.

\begin{proposition}\label{prop:FPRs2Space}
Let $T=\PSp_{2m}(q)$ with $m \geq 3$. Let $x \in G$ have prime order. If $x \in \PGL(V)$, then write $s=\nu(x)$. Let $H$ be the stabiliser in $G$ of a non-degenerate $2$-space. Then \[ \fpr(x, G/H) \leq \left\{ \begin{array}{ll} q^{-2s} + q^{-(2s+2)} + q^{-(2m-2)} + q^{-(2m-1)} & \text{ if $x \in \PGL(V)$ } \\ 2q^{-(2m-1)} & \text{ if $x \not\in \PGL(V)$} \end{array}\right. \]
\end{proposition} 

\begin{proof}
If $x$ is not contained in a $G$-conjugate of $H$, then $\fpr(x,G/H) = 0$. Therefore, let us assume that $x \in H$. Write $L=G \cap \PGL(V)$ and $H_0 = H \cap L$. Hence, $H_0$ is a subgroup of $\GSp_2(q) \times \GSp_{2m-2}(q)$, modulo scalars. \par

\clearpage
\noindent\textbf{Case 1: $x \in L$}

We will adopt the notation of \cite[\SS 3.4]{ref:BurnessGiudici16} for elements of prime order in $L$. Let $x$ have order $r$. Consider the case when $r$ is odd. If $x$ is semisimple, then $x^L$ is determined by the eigenvalues of $x$ on $V$ whereas if $x$ is unipotent, then $x^L$ is described in terms of (but not uniquely determined by) the Jordan form of $x$ on $V$. If $x$ is an involution, then we use the notation of \cite[Table 4.5.1]{ref:GorensteinLyonsSolomon98} if $x$ is semisimple and of Aschbacher and Seitz \cite{ref:AschbacherSeitz76} if $x$ is unipotent. \par

In each case, the description of elements in \cite[Chapter 3]{ref:BurnessGiudici16} allows the splitting of $x^L \cap H_0$ into $H_0$-classes to be easily determined and we verify the bound using the centraliser orders in \cite[Appendix B]{ref:BurnessGiudici16}. For example, if $r=p=2$ and $x=b_s$, then $x^L \cap H_0$ is the union of $x_1^{H_0}$, $x_2^{H_0}$, and $x_3^{H_0}$ where $x_1$, $x_2$ and $x_3$ are the elements $(I_2,b_s)$, $(b_1,a_{s-1})$ and $(b_1,c_{s-1})$ of $\Sp_2(q) \times \Sp_{2m-2}(q)$. So
\begin{align*}
\fpr(x,G/H) = \fpr(x,L/H_0) = \frac{|H_0|}{|L|} \sum_{i=1}^{3} \frac{|C_{L}(x_i)|}{|C_{H_0}(x_i)|}  \leq \frac{1}{q^{2s}} + \frac{1}{q^{2m-1}} + \frac{1}{q^{2m+s-1}}.
\end{align*}

For another example, suppose that $r \not\in \{ p, 2 \}$, so $x$ is a semisimple element of odd order. By \cite[Prop. 3.4.3]{ref:BurnessGiudici16}, a lift of $x$ is $L$-conjugate to the block diagonal matrix $[M_1,\dots,M_d,I_{\ell}]$ where, for some even $k$, the matrices $M_1, \dots, M_d$ either each act irreducibly on a non-degenerate $k$-space $U_i$, or preserve the decomposition of a non-degenerate $k$-space $U_i$ into two totally isotropic $\frac{k}{2}$-spaces, acting irreducibly on both. Now let $h \in H$ be $G$-conjugate to $x$. Then $h$ lifts to $(M,N) \in \GSp_2(q) \times \GSp_{2m-2}(q)$. If $M = I_2$, then $\ell \geq 2$ and $h$ is $H_0$-conjugate to $x_0$, an element lifting to $(I_2, [M_1,\dots,M_d,I_{\ell-2}])$. If $M \neq I_2$, then let $\l \in \K$ be a non-trivial eigenvalue of $M$. Hence, $\l$ is an eigenvalue of $M_i$ for some $i$. Since the set of eigenvalues of $M$ is closed under the map $\mu \mapsto \mu^q$, it must be that $k=2$ and $M=M_i$. Therefore, $h$ is $H_0$-conjugate to $x_i$, an element lifting to $(M_i, [M_1,\dots,M_{i-1},M_{i+1},\dots,M_d,I_{\ell}])$. So, if $\ell = 0$ and $k > 2$, then $x^L \cap H_0 = \emptyset$; if $\ell \geq 2$ and $k > 2$, then $x^L \cap H_0 = x_0^{H_0}$; and if $k = 2$, then $x^L \cap H_0 = x_a^{H_0} \cup \cdots \cup x_d^{H_0}$ where $a=0$ if $\ell \geq 2$ and $a = 1$ if $\ell = 0$. The result follows from the centralisers in \cite[Appendix B]{ref:BurnessGiudici16}. \par

\vspace{5pt}

\noindent\textbf{Case 2: $x$ is a field automorphism}

In this case, $|x^G|$ is at most the number of elements of order $r$ in $Tx$. So $|x^G \cap H|$ is at most the number of elements of order $r$ in $Tx \cap H = H_0x$. By \cite[Prop. 4.9.1(d)]{ref:GorensteinLyonsSolomon98}, this is at most $2|x^H|$. So $|x^G \cap H| \leq 2|x^H|$ and
\[\fpr(x,G/H) = \frac{2|H||C_G(x)|}{|G||C_H(x)|} \leq \frac{2|\Sp_{2}(q)||\Sp_{2m-2}(q)|e |\Sp_{2m}(q^{1/r})|e}{|\Sp_{2m}(q)|e |\Sp_{2}(q^{1/r})||\Sp_{2m-2}(q^{1/r})|e} \leq \frac{2}{q^{2m-2}}. \qedhere \]
\end{proof}

The four-dimensional symplectic groups require special attention and we will provide a close to best possible fixed point ratio bound for these groups.
\begin{proposition}\label{prop:FPRsSymp4}
Let $q=p^f$ where $f>1$ and let $G$ be an almost simple group with socle $\PSp_4(q)$. For a maximal non-subspace subgroup $H$ of $G$ and $x \in G$ of prime order \[ \fpr(x,G/H) \leq \frac{4}{q(q-1)}, \] unless $H$ has type $\Sp_2(q) \wr S_2$ or $\Sp_2(q^2)$ and $x$ is an $a_2$ or $t_2$ involution, in which case, \[ \fpr(x,G/H) \leq \frac{q}{q^2-1}. \]

Moreover, we have the following stronger bounds when $q$ is even.
\begin{enumerate}
\item{If $H$ has type $Sz(q)$, then $\fpr(x,G/H) \leq 1/q^2$.}
\item{If $H$ has type $\O^-_2(q^2)$, then $\fpr(x,G/H) \leq 8/q^2(q-1)$.}
\end{enumerate}
\end{proposition}

\begin{proof}
Let $x$ have prime order $r$. We may assume that $x \in H$. By \cite[Tables 8.12--8.14]{ref:BrayHoltRoneyDougal}, the possibilities for the type of $H$ are the following. (Here $l$ is a prime divisor of $e$.)
\begin{itemize}
\item{in all cases:
\begin{tabularx}{0.5\textwidth}{lccc}
\toprule[0.06em]
Type      & $\Sp_4(q^{1/l})$ & $Sz(q)$              & $\PSL_2(q)$ \\[7pt]
Condition &                  & $q$ even \& $f$ odd  & $q$ odd     \\
\bottomrule[0.06em]
\end{tabularx}}
\item{if $G$ does not contain a graph-field automorphism:
\begin{tabularx}{0.5\textwidth}{lcccc}
\toprule[0.06em]
Type      & $\Sp_2(q) \wr S_2$ & $\GL_2(q).2$   & $\Sp_2(q^2)$ & $\GU_2(q)$ \\[7pt]
Condition &                    & $q$ odd        &              & $q$ odd    \\
\bottomrule[0.06em]
\end{tabularx}}
\item{if $G$ contains a graph-field automorphism: 
\begin{tabularx}{0.5\textwidth}{lccc}
\toprule[0.06em]
Type      & $\O_2^+(q) \wr S_2$  & $\O_2^-(q) \wr S_2$  & $\O_2^-(q^2)$ \\
\bottomrule[0.06em]
\end{tabularx}}
\end{itemize}

Write $T=\PSp_4(q)$, $L=G \cap \PGL(V)$ and $H_0 = H \cap L$. Assume that $H$ does not have type $\PSL_2(q)$ since the calculation for this case is in \cite[Prop. 2.22]{ref:Burness074}. \par

\noindent\textbf{Case 1: $x \in L$}

Suppose for now that $H$ does not have type $Sz(q)$. We proceed as in the proof of Proposition~\ref{prop:FPRs2Space}. The splitting of $x^L$ into $H_0$-classes is straightforward to determine, except for involutions when $q$ is odd. For these elements the arguments are often more subtle. We present the example where $q$ is odd, $x$ is an involution and $H$ has type $\GU_2(q)$. \par

Write $H_0 = B\sp\<\psi\>$ where $B$ is the image of $\GU_2(q)$ in $\Sp_4(q)$ modulo scalars, and $\psi$ induces the inverse-transpose map on $B$. As explained in Section~\ref{ssec:PrelimsGroups}, we will denote semisimple elements of $\GL_4(q)$, up to conjugacy, as $[\l_1,\l_2,\l_3,\l_4]$. By \cite[Lemma 5.3.11]{ref:BurnessGiudici16}, the image of $\GU_2(q)$, modulo scalars, is given as $[\l, \mu] \mapsto [\l, \l^q, \mu, \mu^q]$. First suppose that $x \in B$. Following \cite{ref:GorensteinLyonsSolomon98}, there are two classes of involutions in $B$. The $t_1$ class is represented by an element which lifts to $[-1,1]$ and so embeds in $L$ as $[-I_2,I_2]$, a $t_1$ involution of $L$. If $q \equiv 1 \mod{4}$, then the second class is represented by $t_1'$, which lifts to $[\xi,\xi^{-q}]$, where $\xi$ has order $4$ in $\F_{q^2}^{\times}$. In this case, $\xi \in \F_q$ and so $[\xi,\xi^{-q}]=[\xi,\xi^{-1}]$ embeds in $L$ as $[\xi I_2, \xi^{-1} I_2] \in L$, a $t_2$ involution of $L$. If $q \equiv 3 \mod{4}$, then the second class arises from central involutions $z$ which lift to $[\l,\l]$, where $\l \in \F_{q^2}^{\times}$ has order $4$. Since $\l \not\in \F_q$, $z$ embeds in $L$ as a $t_2'$ involution. Now suppose that $x \in H_0 \setminus B$. Then $x$ lifts to $A\psi$ such that $(A\psi)^2 \in \{I,-I\}$. That is, $A$ is either symmetric or skew-symmetric. Moreover, $x$ has a 1-eigenvector, and hence embeds as $t_1 \in L$, if and only if $A$ is skew-symmetric. So we have determined how $x^L \cap H_0$ splits into $H_0$ classes, and the result follows as in the proof of Proposition~\ref{prop:FPRs2Space}. For example, if $x$ is a $t_1$ involution and $L=\PSp_4(q)$ then, since there are $q+1$ skew-symmetric matrices in $\GU_2(q)$, \[ \fpr(x,G/H) = \frac{|x^L \cap H_0|}{|x^L|} = \frac{|\Sp_2(q)|^2}{|\PSp_4(q)|}\left( \frac{|\GU_2(q)|}{2|\GU_1(q)|^2} + \frac{q+1}{2} \right) = \frac{1}{q^2}. \]

Now consider the case where $H$ has type $Sz(q)$. By \cite[Prop. 3.52]{ref:Burness072}, either $x=c_2$ or $x = [\l_1,\l_1^{-1},\l_2,\l_2^{-1}]$ for $\l_1 \neq \l_2$. For the former case, we use the fact that $|x^T \cap H_0|$ is at most $(q-1)(q^2+1)$, the number of involutions in $Sz(q)$. In the latter case, the bound $|x^T \cap H_0| \leq |H_0|$ suffices. \par

The stronger bound for the subgroup of type $\O^-_2(q^2)$ is obtained by observing that, in this case, $H_0$ does not contain any involutions of type $a_2$ or $b_1$ (see \cite[Prop. 5.9.2]{ref:BurnessGiudici16}, for example). \par

\vspace{5pt}

\noindent\textbf{Case 2: $x$ is a field automorphism}

Assume that if $H$ has type $\Sp_4(q^{1/l})$ then $r \neq l$ and if $H \in \C_2$ or $H \in \C_3$ then $r \neq 2$ (see Table~\ref{tab:GeometricSubgroups}). The calculations in these cases are similar and we will present an example below. If these conditions are not satisfied, then the situation is slightly more complicated. We will demonstrate how to handle this when $r=2$ and $H \in \C_2$ then outline the other cases. \par

Consider the case where $H$ has type $\Sp_2(q) \wr S_2$. Let $H_0 = B\sp\<\pi\>$ where $B \leq H_0$ is the index two subgroup of type $\Sp_2(q) \times \Sp_2(q)$. By \cite[Prop. 3.4.15]{ref:BurnessGiudici16}, we may assume that $x$ is a power of the standard field automorphism. Moreover, we may choose $\pi$ such that $\pi$ and $x$ commute. Since $|x^G|$ is at most the number of elements of order $r$ in $Tx$, $|x^G \cap H|$ is at most the number of elements of order $r$ in $Tx \cap H = H_0x = Bx \cup B\pi x$. If $r \neq 2$, then, since $\pi$ has order two and commutes with $x$, no element of $B\pi x$ has order $r$. In this case, $|x^G \cap H|$ is at most the number of elements of order $r$ in $Bx$ which, by \cite[Prop. 4.9.1(d)]{ref:GorensteinLyonsSolomon98}, is at most $2|x^H|$. If $r=2$, then the previous argument gives the number of involutions in $Bx$, so it remains to determine the number of involutions in $B\pi x$. Let $g\pi x \in B\pi x$ be an involution. Suppose that $g$ lifts to $[M,N] \in \GSp_2(q) \times \GSp_2(q)$. Then for $\l \in \F_q$, \[ \l[I,I] = ([M,N]\pi x)^2 = [M,N][M,N]^{\pi x} = [MN^x,NM^x]. \] Hence, $\l \in \{1,-1\}$ and $N=\l M^{-x}$. So there are at most $2|\Sp_2(q)| = 2q(q^2-1)$ involutions in $B\pi x$. The bound follows. \par

Let us now remark on the remaining subtleties. First, if $r=2$ and $H$ has type $\Sp_2(q^2)$ or $\GU_2(q)$, then $\fpr(x,G/H)=0$. To see this, suppose that $G = \PSp_4(q)\sp\<\s\>$ and that $H = \PSp_2(q^2)\sp\<\t\>$ where $\s$ is a field automorphism of $\PSp_4(q)$ of order $e$ and $\t$ is a field automorphism of $\PSp_2(q^2)$ of order $2e$; the other cases are similar. If $x \not\in \PSp_2(q^2)$ is an involution, then $x = g\t^e$ for some $g \in \PSp_2(q^2)$. However, $g\t^e \in \PSp_2(q^2)\sp\<\t^e\> = H \cap L$ and so $x \in L$: a contradiction. Second, let $H$ have type $\Sp_{2m}(q^{1/l})$ with $r = l$. For $S \subseteq G$, let $i_r(S)$ be the number of elements of $S$ of order $r$. Although $|x^G \cap H| = i_r(H_0x)$, we cannot argue that $i_r(H_0x) = |x^H|$, as we did above, since $x$ commutes with $H_0$. Therefore, we need to explicitly bound $i_r(H_0x) \leq 1 + i_r(H_0)$. If $r \geq 5$ then the bound $i_r(H_0) \leq |H_0|$ suffices, and if $r \in \{2,3\}$ then we use the bounds from \cite[Prop. 1.3]{ref:LawtherLiebeckSeitz02}. 

\vspace{5pt}

\noindent\textbf{Case 3: $x$ is a graph-field automorphism}

In this case $r = p = 2$. First, if $H_0 = \Sp_4(q^{1/l})$, then we argue as for field automorphisms. Second, if $H_0 = Sz(q)$, then, as above, $x$ commutes with $H_0$ and we need the result that $i_2(Sz(q)) = (q-1)(q^2+1)$.  Finally, if $H$ has type $\O^{\e}_2(q) \wr S_2$ or $\O^{-}_2(q^2)$, then, since $H$ is a split extension of $H_0$ by a cyclic group of order $2e=|H:H_0|$, there are at most $|H|/e = 2|H_0|$ elements of order 2 in $H$. The bound $|x^G \cap H| \leq 2|H_0|$ suffices.
\end{proof}


\section{Proof of the Main Results}\label{sec:Proof}
In this final section we will prove Theorems \ref{thm:MainResult}--\ref{thm:MainUpper}. Recall the definitions of $\T$ and $\A$ from \eqref{eq:DefT} and \eqref{eq:DefA}. For this entire section, fix $G = \<T,\th\> \in \A$ and assume that $T \neq \PSp_4(2)' \cong A_6$ (see Remark~\ref{rem:A6}). Let $V = \F_q^n$ be the formed space defined in Table~\ref{tab:Groups}. Moreover, let $X$ be the algebraic group defined in \eqref{eq:DefX}, let $\s$ be the Steinberg morphism defined in $\eqref{eq:DefSigmaField}$ or $\eqref{eq:DefSigmaGraphField}$, let $f$ be the Shintani map defined in \eqref{eq:SympShintaniEvenField}--\eqref{eq:OrthShintani} or Proposition~\ref{prop:SympShintaniEvenGraphField} and let $q=q_0^e$. \par

We will follow the probabilistic approach outlined in the introduction. Let us recall two pieces of notation which are central to this approach. For $x,s \in G$, write $\M(G,s)$ for the set of maximal subgroups of $G$ which contain $s$, and write \[ P(x,s) = 1 - \frac{|\{z \in s^G \mid G = \< x,z \>\}|}{|s^G|}. \]

\subsection{Diagonal automorphisms}\label{ssec:ProofDiagonal}
We will begin by proving Theorems~\ref{thm:MainResult}, \ref{thm:MainSharper} and the reverse direction of Theorem~\ref{thm:MainAsymptotic}, in the case where $\th \in \InnDiag(T)$. We need the following lemma.
\begin{lemma}\label{lem:ElementWhichSquares}
Let $d \geq 1$, let $q$ be odd and let $\<\zeta\> = \F_q^{\times}$.
\begin{enumerate}
\item{Let $A \in \Sp_{2d}(q)$ generate a subgroup $\GU_1(q^d)$. Then there exists $C \in \GSp_{2d}(q)$ such that $C^{q-1} = A$ and $\t(C)=\zeta$. In particular, $C \not\in \Sp_{2d}(q)$.}
\item{Let $A \in \SO^{-}_{2d}(q) \leq \SO_{2d+1}(q)$ generate a subgroup $\GU_1(q^d)$. Then $A \not\in \Om_{2d+1}(q)$. }
\end{enumerate}
\end{lemma}

\begin{proof}
First consider (i). Recall that $\mathrm{\Delta U}_1(q^d)$ is the similarity group of a 1-dimensional unitary space over $\F_{q^d}$ with similarity map $\t\:\mathrm{\Delta U}_1(q^d) \to \F_{q^d}$ (see Section~\ref{ssec:PrelimsGroups}). The group \[ H = \{ h \in \mathrm{\Delta U}_1(q^d) \mid \t(h) \in \F_q \} \] is naturally a subgroup of $\GSp_{2d}(q)$. Moreover, $\< A \>$ is the index $q-1$ subgroup of $H$ containing the isometries in $H$. Hence, there is a generator $C$ for $H$ such that $C^{q-1} = A$. Since $C$ generates $H$, we may assume that $\t(C) = \zeta$. Now consider (ii). By \cite[Theorem 4]{ref:ButurlakinGrechkoseeva07}, $\Om_{2d+1}(q_0)$ does not have a maximal torus of order $|\< A \>| = q^d+1$, so $A \not\in \Om_{2d+1}(q_0)$.
\end{proof}

\begin{proposition}\label{prop:InnDiag}
Let $T \in \T$, let $\th \in \InnDiag(T)$ and let $G = \< T, \th \>$.
\begin{enumerate}
\item{In all cases, $u(G) \geq 2$}
\item{If $T = \Om_{2m+1}(q)$, then $u(G) \geq 3$.}
\item{If $T = \PSp_{2m}(q)$, $q$ is odd and $m \geq 3$, then $u(G) \geq 4$.}
\item{In all cases, $u(G) \to \infty$ as $q \to \infty$.}
\item{If $T = \PSp_{2m}(q)$ and $q$ is odd, then $u(G) \to \infty$ as $m \to \infty$.}
\end{enumerate}
\end{proposition}

\begin{proof}
By \cite[Corollary 1.3]{ref:BreuerGuralnickKantor08}, (i) and (ii) hold if $\th=1$. In particular, (i) holds if $q$ is even. Therefore, let us suppose that $q$ is odd to prove parts (i)--(iii). We may exclude the cases covered by the \textsc{Magma} computations listed in Table~\ref{tab:Comp}. \par

For now, let us assume that $m$ is odd if $T=\Om_{2m+1}(3)$. In the proofs of \cite[Prop. 5.10, 5.12, 5.19, 5.20]{ref:BreuerGuralnickKantor08}, by separating into several cases depending on $T$ and $m$, it is shown that for all prime order elements $x \in T$, $P(x,s) < 1/3$, for a suitable choice of semisimple element $s \in T$. In each case, by Lemma~\ref{lem:ElementWhichSquares}, there exists $g \in G \setminus T$ such that $g^2=s$. The proofs that $P(x,s) < 1/3$ each comprise two steps: determining $\M(T,s)$ and computing the fixed point ratios $\fpr(x,T/H_0)$ for all $H_0 \in \M(T,s)$. To determine $\M(T,s)$, the main result of \cite{ref:GuralnickPentillaPraegerSaxl97} is applied and the properties of $T$ and $s$ which are used to eliminate subgroups hold also for $G$ and $g$. Hence, the subgroups in $\M(G,g)$ have the same type and multiplicities as those in $\M(T,s)$. Moreover, the bounds on $\fpr(x,T/H_0)$ for $H_0 \in \M(T,s)$ apply also to $\fpr(x,G/H)$ for $H \in \M(G,s)$ and $x \in G$. Therefore, $P(x,g) < 1/3$ and so $u(G) \geq 3$. In fact, if $m \geq 3$ and $T=\PSp_{2m}(q)$, then, by using the bounds from Proposition~\ref{prop:FPRs} instead of the bounds in \cite{ref:BreuerGuralnickKantor08}, we obtain $P(x,s) < 1/4$ and $P(x,g) < 1/4$. As a result, $u(G) \geq 4$. \par

For $T=\Om_{2m+1}(3)$ with $m$ even, for suitable $s \in T$, it is shown in \cite[Prop. 5.7]{ref:BreuerGuralnickKantor08} that $P(x,s) \leq 1/3$ with equality if and only if $x$ is an involution with $\nu(x)=1$. By Lemma~\ref{lem:ElementWhichSquares}, we can choose $g \in G\setminus T$ and, by arguing as above, we show that $P(x,g) \leq 1/3$ with equality if and only if $x$ is an involution with $\nu(x)=1$. By the argument in \cite[Prop. 5.7]{ref:BreuerGuralnickKantor08}, for all involutions $x_1,x_2,x_3 \in T$ such that $\nu(x_1)=\nu(x_2)=\nu(x_3)=1$, there exists a $G$-conjugate $z$ of $g$ for which $\<x_1,z\> = \<x_2,z\> = \<x_3,z\> = G$. Therefore, $u(G) \geq 3$. \par

Now consider parts (iv) and (v). By \cite[Theorem 1.1]{ref:GuralnickShalev03}, these parts holds when $\th = 1$. In the proof of \cite[Prop. 4.1]{ref:GuralnickKantor00}, it is shown that $P(x,s) \to 0$ as $m \to \infty$ or $q \to \infty$, for suitable $s \in \PSp_{2m}(q)$. By Lemma~\ref{lem:ElementWhichSquares}, there exists  $g \in \PGSp_{2m}(q) \setminus \PSp_{2m}(q)$ such that $P(x,g) \to 0$ as $m \to \infty$ or $q \to \infty$. Very similarly, by the proof of \cite[Prop. 4.1]{ref:GuralnickKantor00}, we can find $g \in \SO_{2m+1}(q) \setminus \Om_{2m+1}(q)$ such that $P(x,g) \to 0$ as $q \to \infty$.
\end{proof}

Therefore, if $\th \in \InnDiag(T)$, then it remains to prove only Theorem~\ref{thm:MainUpper} (which implies the forward direction of Theorem~\ref{thm:MainAsymptotic}). This will be done in Section~\ref{ssec:ProofAsymptotic}. Therefore, in Sections~\ref{ssec:ProofElements}--\ref{ssec:ProofProbMethod} we will assume that $\th \not\in \InnDiag(T)$.

\subsection{Element selection}\label{ssec:ProofElements}
Let $G = \<T,\th\> \in \A$ with $T \neq \PSp_4(2)'$ and $\th \not\in \InnDiag(T)$. Maintain the notation introduced at the opening of Section~\ref{sec:Proof}. In particular, recall that $q=q_0^e$. The goal of this section is to identify an element $t\th \in G$ to represent the conjugacy class with respect to which we will study the uniform spread of $G$. \par

We will introduce notation for the elements which we will use repeatedly. 
\begin{definition}\label{def:ABCD}
Let $d \geq 2$, $W=\F_{q_0}^{2d}$ and $\F_{q_0}^{\times} = \< \a \>$.
\vspace{-2pt}

\begin{description}[labelindent=5pt]
\item[Cases $\mathbf{S}$ and $\mathbf{S_4}$]{\new
\begin{enumerate}[labelindent=5pt]
\item{Let $A_{2d} \in \Sp_{2d}(q_0)$ be a generator of a cyclic subgroup $\GU_1(q_0^d)$.}
\item{Write $W=W_1 \oplus W_2$ where $W_1$ and $W_2$ are totally isotropic $d$-spaces. Let $B_{2d} = [B, B^{-T}] \in \Sp_{2d}(q_0)$ have order $q_0^d-1$ and stabilise the spaces $W_1$ and $W_2$, acting irreducibly on both.}
\item{If $q$ is odd, then let $C_{2d} \in \GSp_{2d}(q_0)$ be such that $C_{2d}^{q_0-1} = A_{2d}$ and $\t(C_{2d}) = \a$ (see Lemma~\ref{lem:ElementWhichSquares}).}
\item{Let $D_{2d} = [\a B, B^{-T}] \in \GSp_{2d}(q_0)$, where $B$ is as in (ii).}
\end{enumerate}}
\item[Case $\mathbf{O}$]{\new
\begin{enumerate}[labelindent=5pt]
\setcounter{enumi}{3}
\item{Assume that $W$ is minus-type. Let $A_{2d} \in \SO^-_{2d}(q_0)$ be a generator of a cyclic subgroup $\GU_1(q_0^d)$.}
\item{Assume that $W$ is plus-type and write $W=W_1 \oplus W_2$ where $W_1$ and $W_2$ are totally isotropic $d$-spaces. Let $B_{2d} = [B, B^{-T}] \in \SO^+_{2d}(q_0)$ have order $q_0^d-1$ and stabilise the spaces $W_1$ and $W_2$, acting irreducibly on both.}
\end{enumerate}}
\end{description}
\end{definition}

Let us record some straightforward properties of the elements defined in Definition~\ref{def:ABCD}.
\begin{lemma}\label{lem:ABCD}
Adopt the notation from Definition~\ref{def:ABCD}.
\begin{enumerate}
\item{$A_{2d}$ has order $q_0^d+1$.}
\item{$A_{2d}$ acts irreducibly on $W$.} 
\item{The eigenvalues of $A_{2d}$ over $\overline{\F}_{q_0}$ are $\l, \l^{q_0}, \dots, \l^{q_0^{2d-1}}$, for some $\l \in \overline{\F}_{q_0}$ of order $q_0^d+1$. Moreover, these eigenvalues are distinct.}
\item{In cases $\mathbf{S}$ and $\mathbf{S_4}$, $C_{\Sp_{2d}(q_0)}(A_{2d}) = \< A_{2d} \>$, and in case $\mathbf{O}$,  $C_{\SO^{-}_{2d}(q_0)}(A_{2d}) = \< A_{2d} \>$.}
\item{$B_{2d}$ has order $q_0^d-1$.}
\item{The eigenvalues of $B_{2d}$ over $\overline{\F}_{q_0}$ are \[ \mu, \mu^{-1}, \mu^{q_0}, \mu^{-q_0}, \dots, \mu^{q_0^{d-1}}, \mu^{-q_0^{d-1}},\] for some $\mu \in \overline{\F}_{q_0}$ of order $q_0^d-1$. Moreover, if $d$ is odd, then these are distinct.}
\item{Assume that $d$ is odd. In cases $\mathbf{S}$ and $\mathbf{S_4}$, $C_{\Sp_{2d}(q_0)}(B_{2d}) = \< B_{2d} \>$, and in case $\mathbf{O}$,  $C_{\SO^{+}_{2d}(q_0)}(B_{2d}) = \< B_{2d} \>$.}
\end{enumerate}
\end{lemma}

\begin{proof}
See \cite[Prop. 3.4.3, 3.5.4, Remarks 3.4.4, 3.5.6]{ref:BurnessGiudici16}.
\end{proof}

Recall that $X$ is the algebraic group defined in \eqref{eq:DefX} and $\s$ is the Steinberg morphism defined in $\eqref{eq:DefSigmaField}$ or $\eqref{eq:DefSigmaGraphField}$. We will define $t\th$ as the preimage, under a Shintani map, of an element $\oly \in X_{\s}$. We need to select $t\th \in G$ in a way which allows us to control the maximal subgroups of $G$ which contain it. Therefore, we will choose $t\th$ such that it has the following two features, which place significant restrictions on its maximal overgroups. First, $t\th$ should not be contained in many reducible subgroups, and second, a power of $t\th$ should have a 1-eigenspace of large dimension in its action on the natural module for $G$. These two conditions will inform our choice of element $\oly \in X_{\s}$. \par

{\renewcommand{\arraystretch}{1.1}
\begin{tabularx}{0.8\textwidth}{ccccc}
\toprule[0.12em]
Case           & $q$  & $\th$    & $y$                  & Condition \\
\midrule
$\mathbf{S}$   & even & $\p^i$   & $[A_2,A_{2m-2}]$     & $m$ odd              \\
               &      &          & $[A_2,B_{2m-2}]$     & $m$ even             \\[3pt]
               & odd  & $\p^i$   & $[A_2,A_{2m-2}]$     & $m$ odd              \\
               &      &          & $[A_2,B_{2m-2}]$     & $m$ even             \\[3pt]
               &      & $\d\p^i$ & $[C_2,C_{2m-2}]$     & $m$ odd              \\
               &      &          & $[C_2,D_{2m-2}]$     & $m$ even             \\[3pt]
$\mathbf{O}$   & odd  & $\p^i$   & $[A_2,A_{2m-2},1]^2$ & $m$ odd              \\
               &      &          & $[A_2,B_{2m-2},1]^2$ & $m$ even             \\[3pt] 
               &      & $\d\p^i$ & $[A_2,A_{2m-2},1]$   & $m$ odd              \\
               &      &          & $[A_2,B_{2m-2},1]$   & $m$ even             \\[3pt]
$\mathbf{S_4}$ & even & $\p^i$   & $A_4$                &                      \\[3pt]
               &      & $\r^j$   & $A_4^{\ell}$         &                      \\[3pt]  
               & odd  & $\p^i$   & $A_4$                &                      \\[3pt]
               &      & $\d\p^i$ & $C_4$                &                      \\[3pt]          
\bottomrule[0.12em]
\caption{The element $t\th \in G$ satisfies $f(t\th) = \oly$.}\label{tab:Elements}
\end{tabularx}}
\vspace{-5.5pt}

Table~\ref{tab:Elements} partitions the possibilities for $G$ into several cases, and, in each case, an element $y$ is given. We will now define these elements more precisely, and we will verify that in each case $\oly$ (the image of $y$ modulo scalars) is contained in the image of the coset $T\th$ under the relevant Shintani map.  \par

In case $\mathbf{S}$, the element $y \in \GSp_{2m}(q_0)$ is a block diagonal matrix preserving a decomposition $V = U_1 \oplus U_2$, where $U_1$ and $U_2$ are non-degenerate subspaces of dimensions $2$ and $2m-2$, respectively. If $\th = \p^i$ then $\oly \in \PSp_{2m}(q_0)$, and if $\th = \d\p^i$ (so $q_0$ is odd) then $\overline{y} \in \PGSp_{2m}(q_0) \setminus \PSp_{2m}(q_0)$ since, by Lemma~\ref{lem:ElementWhichSquares}(i), $\t(y) = \a$, a non-square in $\F_{q_0}$. Therefore, in each case, by Proposition~\ref{prop:SympShintaniOdd}, $\oly \in f(T\th)$. \par

In case $\mathbf{O}$, the element $y \in \SO_{2m+1}(q_0)$ is a block diagonal matrix preserving a decomposition $V = U_1 \oplus U_2 \oplus U_3$ where $U_1$, $U_2$ and $U_3$ are non-degenerate subspaces of dimensions $2$, $2m-2$ and $1$, respectively. Moreover, $U_1$ is plus-type and $U_2$ is $\e$-type where $\e=(-)^m$. If $\th = \p^i$ then $y \in \Om_{2m+1}(q_0)$ since $[A_2,A_{2m-2},1] \in \SO_{2m+1}(q_0)$ and $[A_2,B_{2m-2},1] \in \SO_{2m+1}(q_0)$. However, if $\th = \d\p^i$ then $y \in \SO_{2m+1}(q_0) \setminus \Om_{2m+1}(q_0)$, by Lemma~\ref{lem:ElementWhichSquares}. Therefore, by Proposition~\ref{prop:OrthShintani}, $\oly \in f(T\th)$.   \par

In case $\mathbf{S_4}$, if $\th = \p^i$ then $y \in \PSp_4(q)$, and if $\th = \d\p^i$ then $y \in \PGSp_4(q) \setminus \PSp_4(q)$. Before defining the element $y$ when $\th = \r^j$, let us record a useful number theoretic notion. \par

For positive integers $a,k$, we say that $r$ is a \emph{primitive prime divisor (ppd)} of $a^k-1$ if $r$ divides $a^k-1$ but $r$ does not divide $a^i-1$ for $1 \leq i < k$. The following is a theorem of Zsigmondy \cite{ref:Zsigmondy82}.
\begin{theorem}\label{thm:PPD}
If $k \geq 2$ and $(a,k) \not\in \{ (2,6) \} \cup \{ (2^l-1,2) \mid l \in \Nat\}$, then $a^k-1$ has a primitive prime divisor.
\end{theorem}

Let $\th = \r^j$. By Lemma~\ref{lem:ABCD}(i), $A_4$ has order $q_0^4-1$. By Theorem~\ref{thm:PPD}, let $r$ be a ppd of $q_0^4 -1$ and let $\ell$ be a positive integer such that $A_4^{\ell}$ has order $r$. Since $r$ divides $|Sz(q_0)|=q_0^2(q_0-1)(q_0^2+1)$, the subgroup $Sz(q_0)$ contains an element of order $r$. Therefore, we may assume that $y \in Sz(q_0) \leq \Sp_4(q_0)$. Hence, by Proposition~\ref{prop:SympShintaniEvenGraphField}, $\oly \in f(T\th)$. \par
 
To summarise, for each row of Table~\ref{tab:Elements}, we have verified that $\oly \in f(T\th)$. Therefore, we define $t\th \in G$ as an element such that $f(t\th) = \oly$.

\subsection{Maximal subgroups}\label{ssec:ProofMaximalSubgroups}
Let $G = \<T,\th\> \in \A$ with $T \neq \PSp_4(2)'$ and $\th \not\in \InnDiag(T)$. Maintain the notation introduced at the opening of Section~\ref{sec:Proof}. For this section, fix $t\th$ as the element defined in Table~\ref{tab:Elements}. The aim of this section is to study $\M(G,t\th)$, the set of maximal subgroups of $G$ which contain $t\th$. The main result is the following.

\begin{proposition}\label{prop:MaximalSubgroups}
The maximal subgroups of $G$ which contain $t\th$ are listed in Table~\ref{tab:MaximalSubgroups}, where $m(H)$ is an upper bound on the multiplicity of the subgroups of type $H$ in $\M(G,t\th)$.
\end{proposition}

Before proving Proposition~\ref{prop:MaximalSubgroups}, we will first prove two results on the multiplicities of subgroups in $\M(G,t\th)$. Recall that $G_1=X_{\s^e}\sp\<\s\>$.
\begin{proposition}\label{prop:IsomorphismIsConjugacy} 
Maximal geometric subgroups of $G$ of the same type are $G$-conjugate except for subfield subgroups over $\F_{q^{1/2}}$, in which case there are at most two $G$-classes but exactly one $G_1$-class.
\end{proposition}

\begin{proof}
Note that $q$ is not prime since $\th \not\in \InnDiag(T)$. If $n \leq 12$, then the result follows from the tables in \cite[Chapter 8]{ref:BrayHoltRoneyDougal}. Now suppose that $n \geq 13$. Let $H$ be a maximal geometric subgroup of $G$. By \cite[Theorem 3.1.1]{ref:KleidmanLiebeck}, the subgroups of the same type as $H$ are $\Aut(T)$-conjugate to $H$. Moreover, the group $\Aut(T)/T$ acts on $\{ H_1, \dots, H_c \}$, a set of representatives of the $T$-classes of subgroups of $\Aut(T)$ of the same type as $H$. Let $\pi\:\Aut(T)/T \to S_c$ be the permutation representation of this action. By \cite[Tables 3.5C \& 3.5D]{ref:KleidmanLiebeck}, $c=1$ and the $G$-classes of subgroups are precisely the $\Aut(T)$-classes, except for the exceptional case in the statement.  In this case, by \cite[Tables 3.5C \& 3.5D]{ref:KleidmanLiebeck}, $c=2$ and the $\Aut(T)$-class splits into two $T$-classes. By \cite[Table 3.5G]{ref:KleidmanLiebeck}, $\d$ is not contained in the kernel of $\pi$. Therefore, $\d$ permutes the two $T$-classes. Since $\d \in G_1$, all subfield subgroups over $\F_{q^{1/2}}$ are $G_1$-conjugate.
\end{proof}

We will now present a consequence of Proposition~\ref{prop:CentraliserBound} which provides a general bound on the multiplicities of subgroups in $\M(G,t\th)$.
\begin{corollary}\label{cor:MultiplicityCentraliser}
Let $H$ be a maximal subgroup of $G$ and let $t\th \in G$ be the element defined in Table~\ref{tab:Elements}. Then there are at most $N$ subgroups of type $H$ in $\M(G,t\th)$, where \[
N = 
\left\{ 
\begin{array}{ll}
(q_0+1)(q_0^{m-1}+1)    & \text{in case $\mathbf{S}$} \\
q_0(q_0+1)(q_0^{m-1}+1) & \text{in case $\mathbf{O}$} \\
q_0^2+1                 & \text{in case $\mathbf{S_4}$ and $\th$ is a field automorphism} \\
q_0 + \sqrt{2q_0} + 1   & \text{in case $\mathbf{S_4}$ and $\th$ is a graph-field automorphism} \\
\end{array}
\right.
\]
\end{corollary}

\begin{proof}
By Proposition~\ref{prop:IsomorphismIsConjugacy}, the subgroups of type $H$ are $G_1$-conjugate. Therefore, the number of subgroups of type $H$ in $\M(G,t\th)$ is at most $|C_{X_{\s}}(f(t\th))|$, by Proposition~\ref{prop:CentraliserBound}. \par

First consider case $\mathbf{O}$, and cases $\mathbf{S}$ and $\mathbf{S_4}$ when $q$ is even. Here, $X_{\s}$ is a matrix group and $f(t\th)$ is $X$-conjugate to $y$. Therefore, by Lemma~\ref{lem:ABCD}, \[ |C_{X_{\s}}(f(t\th))| = |C_{X_{\s}}(y)| \leq N.\] \par

Now consider cases $\mathbf{S}$ and $\mathbf{S_4}$ when $q$ is odd. Thus, $T=\PSp_{2m}(q_0)$, $X_{\s}=\PGSp_{2m}(q_0)$ and $y \in \Sp_{2m}(q_0)$ with $|C_{\Sp_{2m}(q_0)}(y)| \leq N$. By considering the eigenvalues of $y$, we see that $y$ is not $T$-conjugate to $-y$. Hence, $|C_{\Sp_{2m}(q_0)}(y)|=2|C_{T}(\oly)|$. Therefore, \[ |C_{X_{\s}}(f(t\th))| \leq |C_{X_{\s}}(y)| \leq 2|C_{T}(\oly)| = |C_{\Sp_{2m}(q_0)}(y)| \leq N. \qedhere \]
\end{proof}

We will now prove Proposition~\ref{prop:MaximalSubgroups}. Let $H \in \M(G,t\th)$. If $T \leq H$, then $\th \in H$, since $t\th \in H$. Thus $H=G$: a contradiction. Hence, $T\not\leq H$, so, by \cite[Main Theorem]{ref:KleidmanLiebeck}, $H$ lies in one of the geometric families $\C_1, \dots, \C_8$ or is an almost simple irreducible group in the $\S$ collection. We will prove Proposition~\ref{prop:MaximalSubgroups} in three parts, considering reducible, imprimitive and primitive subgroups in turn. We begin with the reducible subgroups.

\clearpage
{\renewcommand{\arraystretch}{1.3}
\begin{tabularx}{\textwidth}{ccccc}
\toprule[0.12em]
Case           & $\th$ & Type of $H$                        & $m(H)$                              & Conditions  \\
\midrule
$\mathbf{S}$   & any   & $\Sp_2(q) \times \Sp_{2m-2}(q)$    & $1$                                 &             \\
               &       & $P_{m-1}$                          & $2$                                 & $m$ even    \\
               &       & $\Sp_{2}(q) \wr S_m$               & $1$                                 &             \\
               &       & $\GL_{m}(q).2$                     & $2^{(m-1,e)}$                       & $q$ odd  \& *   \\
               &       & $\Sp_{m}(q) \wr S_2$               & $\frac{1}{2}\binom{m}{\frac{m}{2}}$ & $m$ even    \\[12.5pt]
               &       & $\Sp_{2m}(q^{1/l})$                & $\left\{ \begin{array}{ll} e^2 & \text{if $l=e$} \\  q_0^m+q_0^{m-1}+q_0+1 & \text{if $l \neq e$}  \end{array} \right.$ & \\[7pt]
               &       & $\O^{\e}_{2m}(q)$                  & 1                                   & $q$ even    \\[5.5pt]
\midrule
$\mathbf{O}$   & any   & $\O^{\e}_{2m}(q)$                  & $1$                                 &             \\
               &       & $\O^{\e}_2(q) \times \O_{2m-1}(q)$ & $1$                                 &             \\
               &       & $\O_3(q) \times \O^{\e}_{2m-2}(q)$ & $1$                                 &             \\
               &       & $P_{m-1}$                          & $2$                                 & $m$ even    \\[5.5pt]
               &       & $\O_{2m+1}(q^{1/l})$               & $\left\{ \begin{array}{ll} e^3 & \text{if $l=e$} \\  q_0^{m+1}+q_0^m+q_0^2+q_0 & \text{if $l \neq e$}  \end{array} \right.$ & \\[12.5pt]
\midrule
$\mathbf{S_4}$ & not g-f & $\Sp_{2}(q) \wr S_2$             & $1$                                 & $e$ even    \\
               &       & $\GL_{2}(q).2$                     & $q_0^2+1$                           & $q$ odd     \\
               &       & $\Sp_{2}(q^2)$                     & $1$                                 & $e$ odd     \\
               &       & $\GU_{2}(q)$                       & $q_0^2+1$                           & $q$ odd     \\[5.5pt]
               &       & $\Sp_{4}(q^{1/l})$                 & $\left\{ \begin{array}{ll} e & \text{if $l=e$} \\  q_0^2+1 & \text{if $l \neq e$}  \end{array} \right.$ & \\[11pt]
               &       & $\O^{\e}_4(q)$                     & $1$                                 & $q$ even    \\
               &       & $Sz(q)$                            & $q_0^2+1$                           & $q$ even    \\
               &       & $\SL_2(q)$                         & $q_0^2+1$                           & $q$ odd     \\[5.5pt]
\midrule
$\mathbf{S_4}$ & g-f   & $\O_2^{+}(q) \wr S_2$              & $q_0+\sqrt{2q_0}+1$                 & $e \neq 1$  \\
               &       & $\O_2^{-}(q) \wr S_2$              & $q_0+\sqrt{2q_0}+1$                 & $e \neq 1$  \\
               &       & $\O_2^{-}(q^2)$                    & $q_0+\sqrt{2q_0}+1$                 &             \\[5.5pt]
               &       & $\Sp_{4}(q^{1/l})$                 & $\left\{ \begin{array}{ll} e & \text{if $l=e$} \\  q_0+\sqrt{2q_0}+1  & \text{if $l \neq e$}  \end{array} \right.$ & \\[16.5pt]
               &       & $Sz(q)$                            & $\left\{ \begin{array}{ll} 1 & \text{if $e=1$} \\  q_0+\sqrt{2q_0}+1  & \text{if $e \neq 1$}  \end{array} \right.$ & \\[11pt]
\bottomrule[0.12em]
\captionsetup{width=\textwidth}
\caption{Description of $\M(G,t\th)$ \\[2pt] (g-f = graph-field; $l$ is a prime divisor of $e$; $\e \in \{+,-\}$; * for odd $m$, $\frac{2m-2}{(2m-2,e)}$ is odd)} \label{tab:MaximalSubgroups}
\end{tabularx}}

\clearpage
\begin{proposition}\label{prop:MaximalSubgroupsReducible}
Proposition~\ref{prop:MaximalSubgroups} is true for reducible subgroups.
\end{proposition}

\begin{proof}
First consider parabolic subgroups and, for cases $\mathbf{S}$ and $\mathbf{S_4}$, the stabilisers of non-degenerate subspaces. (That is, let us postpone the study of stabilisers of non-degenerate subspaces in case $\mathbf{O}$.) By Proposition~\ref{prop:ShintaniTransfer}, the maximal reducible subgroups of $G$ which contain $t\th$ correspond to the maximal reducible subgroups of $X_{\s}$ which contain $f(t\th)$. Since $f(t\th)$ is $X$-conjugate to $\oly$, the result follows by inspecting the maximal reducible overgroups of $\oly$ in $X_{\s}$. \par

Now consider stabilisers of non-degenerate subspaces in case $\mathbf{O}$. These subgroups are disconnected, so we alter our approach slightly. Let $L=\< \SL_n(q),\th \>$ and $Y = \SL_n(\K)$. Observe that $t\th \in G \leq L$ and $f(t\th) \in X_{\s} \leq Y_{\s}$. Therefore, by considering the maximal overgroups of $\oly$ in $X_{\s}$, \cite[Corollary~2.15]{ref:BurnessGuest13} (the analogue of Proposition~\ref{prop:ShintaniTransfer} in the linear case) demonstrates that $t\th$ is contained in exactly one subgroup of $L$ of types $\SL_{2m}(q)$, $\SL_{2}(q) \times \SL_{2m-1}(q)$ and $\SL_{3}(q) \times \SL_{2m-2}(q)$. In particular, the only possibilities for maximal reducible subgroups of $G$ which contain $t\th$ are those listed in Table~\ref{tab:MaximalSubgroups}.
\end{proof}

Before considering the imprimitive subgroups, we state the following elementary lemma.
\begin{lemma} \label{lem:MultiplicitiesGeneralFixed}
Let $G$ be a finite group and let $H$ be a self-normalising subgroup of $G$. Then for all $x \in G$, the number of $G$-conjugates of $H$ which contain $x$ is \[ \frac{|G|}{|H|}\frac{|x^G \cap H|}{|x^G|}.\]  
\end{lemma}

A subgroup of $\GL_n(q)$ is \emph{imprimitive} if it stabilises a direct sum decomposition \[ \F_q^n = V_1 \oplus \cdots \oplus V_k,\]  for some $k > 1$, possibly permuting the summands. It is \emph{primitive} otherwise. 
\begin{proposition}\label{prop:MaximalSubgroupsImprimitive}
Proposition~\ref{prop:MaximalSubgroups} is true for irreducible imprimitive subgroups.
\end{proposition}

\begin{proof}
Consider the cases $\mathbf{S}$ and $\mathbf{O}$. Recall that $n$ is either $2m$ or $2m+1$, depending on the case, and $V=\F_q^n$. Let $H \leq G$ be a maximal imprimitive subgroup of $G$ containing $t\th$. Then $H$ is the stabiliser in $G$ of the direct sum decomposition 
\begin{equation}
V=V_1 \oplus \cdots \oplus V_k \label{eq:directsum}
\end{equation}
where $k \geq 2$ divides $n$ and $\dim V_i = n/k$ for all $i \in \{1, \dots, k\}$. For the maximality of $H$, we require that either each $V_i$ is non-degenerate or that, in case $\mathbf{S}$, $k = 2$ and each $V_i$ is totally isotropic. In either case $\dim V_i \geq 2$ and, consequently, $k \leq m$. (That $\dim V_i \neq 1$ in the case $\mathbf{O}$ follows by the maximality of $H$ since $q$ is not prime; see \cite[Table 3.5D]{ref:KleidmanLiebeck}.) \par

By construction, a suitable power of $t\th$ lifts to an element $x$ of order $r$, a ppd of $q_0^{\b}-1$ where $\b = (2m-2)/(m,2)$. (Since $\b \not\in \{2,6\}$, a ppd of $q_0^{\b}-1$ exists by Theorem~\ref{thm:PPD}.) We claim that $x$ stabilises each summand in \eqref{eq:directsum}. Suppose that $x$ induces a non-trivial permutation $\pi$ on the summands. Then $\pi$ is a non-trivial product of $r$-cycles, so $r \leq k$. However, $r$ is a ppd of $q_0^{\b}-1$, so $\b$ divides $r-1$ and so $\b+1 \leq r$. Hence, $\b+1 \leq r \leq k \leq m$. If $m$ is odd, then this is a contradiction. If $m$ is even, then $r=k=m$. However, $r$ is odd and $m$ is even: another contradiction. Therefore, $x$ stabilises each summand in \eqref{eq:directsum}. \par

For $K=\K$, let $\olV = \< u_1, \dots, u_n \>_{K}$ and extend the semilinear action of $G$ on $V$ to an action on $\olV$ by defining, for each $g \in G \cap \GL(V)$ and $\a_1, \dots, \a_n \in K$, \[ (\a_1 u_1 + \cdots + \a_n u_n) g\s = \a_1^{q_0}(u_1g) + \cdots + \a_n^{q_0}(u_ng).\] Then the decomposition in \eqref{eq:directsum} gives rise to the corresponding decomposition 
\begin{equation}
\olV = \olV_1 \oplus \cdots \oplus \olV_k. \label{eq:directsumtensor}
\end{equation}

Recall that $f(t\th)$ lifts to the element $y$ in Table~\ref{tab:Elements}. We will show that $y$ stabilises each summand in \eqref{eq:directsum}. Suppose that $x$ acts non-trivially on $V_i$ and $1 \neq \mu \in K$ is an eigenvalue of $x$ with $\mu$-eigenvector $v \in \olV_i$. Since $x$ and $y$ commute, \[ (vy)x=(vx)y=(\mu v)y = \mu(vy).\] That is, $vy$ is a $\mu$-eigenvector of $x$. However, all non-trivial eigenvalues of $x$ have multiplicity one, so $vy \in \olV_i$. Since $y$ preserves the decomposition \eqref{eq:directsumtensor}, $y$ stabilises $\olV_i$. However, $V$ is $y$-stable, so $y$ stabilises $\olV_i \cap V = V_i$. Since the 1-eigenspace of $x$ is at most 3-dimensional and $\dim V_i \geq 2$, $x$ acts non-trivially on at least $k-1$ summands. Therefore, $y$ stabilises at least $k-1$ summands and, hence, all $k$ summands. \par

Now we will find subspaces which are stabilised by $t\th$. By Lemma~\ref{lem:ABCD}, the eigenvalue set of $y$ is $\{\l_1, \l_1^{q_0}, \l_2, \l_2^{q_0}, \dots, \l_2^{q_0^{2m-3}}\}$. Let $\olV_i$ contain the $\l_1$-eigenspace of $y$ and $\olV_j$ contain the $\l_1^{q_0}$-eigenspace of $y$. Since $y$ and $t\th$ commute, if $v \in \olV_i$ is a $\l_1$-eigenvector for $y$, then \[ (vt\th)y = (vy)(t\th) = (\l_1 v)(t\th) = \l_1^{q_0}(vt\th),\] so $vt\th$ is a $\l_1^{q_0}$-eigenvector for $y$. However, $\l_1^{q_0}$ has multiplicity one so $vt\th \in \olV_j$. Similarly, if $w \in \olV_j$ is a $\l_1^{q_0}$-eigenvector, then $w\th$ is a $\l_1$-eigenvector, so $wt\th \in \olV_i$. Thus, since $y$ preserves \eqref{eq:directsumtensor}, $t\th$ stabilises $\olV_i + \olV_j$, and, since $V$ is $t\th$-stable, $t\th$ stabilises $V_i+V_j$. \par

First consider case $\mathbf{S}$. If $i \neq j$, then $t\th$ stabilises $V_i \oplus V_j$, so, by Proposition~\ref{prop:MaximalSubgroupsReducible}, $2\dim V_i = 4m/k \in \{2,2m-2,2m\}$. However, $m \geq 3$ and $2m/k$ divides $2m$, so $k=2$. Similarly, if $i=j$ then $t\th$ stabilises $V_i$, so $\dim V_i = 2m/k \in \{2,m-1,m+1,2m-2\}$. Since $2m/k$ divides $2m$, $\dim V_i = 2$. By a similar line of reasoning, in case $\mathbf{O}$, we must have that $\dim V_i = 3$. Then $y$ is a block diagonal matrix $[M_1, \dots,M_k]$ where $M_i \in \SO_{3}(q)$. Hence, $M_i$ has eigenvalues $\l_i, \l_i', 1$, contradicting 1 being an eigenvalue of $y$ of multiplicity 1. \par

To summarise, we have established that in case $\mathbf{O}$ no imprimitive irreducible subgroups arise, and in case $\mathbf{S}$ either $k = 2$ or $k = m$. We will now obtain an upper bound on the number of such subgroups in case $\mathbf{S}$. To do this, we will use Lemma~\ref{lem:MultiplicitiesGeneralFixed}. \par

Let $H$ be the stabiliser in $G$ of the decomposition \eqref{eq:directsum} and let $B$ be the subgroup of $H$ stabilising each summand. Recall $\oly \in B$ since $y$ stabilises each summand. \par \vspace{4pt}

\noindent\textbf{Case 1: $k=m$}

With respect to a suitable basis, $y$ is a block diagonal matrix $[M_1, \dots, M_m]$ where $M_i \in \GSp_2(q)$. Since the eigenvalues of $y$ are distinct, for $\e_i \in \{+,-\}$, \[ |C_G(\oly)| = |\GL^{\e_1}_1(q)|\cdots|\GL^{\e_m}_1(q)| = |C_B(\oly)| = |C_H(\oly)|. \] Moreover, $\oly^G \cap H$ splits into $m!$ $B$-classes (corresponding to reordering $M_1, \dots, M_m$), which are fused in $H$. So $\oly^G \cap H = \oly^H$, and, by Lemma~\ref{lem:MultiplicitiesGeneralFixed}, $\oly$ is contained in exactly one $G$-conjugate of $H$. Hence, $t\th$ is contained in at most one $G$-conjugate of $H$. \par \vspace{2pt}

\noindent\textbf{Case 2: $k=2$ and $V_1,V_2$ are non-degenerate}

Here $m$ is even. Therefore, $y = [A_2,B_{2m-2}]$. Since $\oly \in B$, with respect to a suitable basis, $y = [M,N]$ where $M,N \in \GSp_m(q)$. By Lemma~\ref{lem:ABCD}, the eigenvalue set of $y$ is $\{ \l_1, \l_1^{q_0},\l_2, \l_2^{-1}, \dots,\l_2^{q_0^{m-2}}, \l_2^{-q_0^{m-2}} \}$. Since the eigenvalues of $M$ are closed under taking inverses and since $\l_1^{q_0} = \l_1^{-1}$, we may assume that $\l_1$ and  $\l_1^{q_0}$ are eigenvalues of $M$. \par

Let $d = (m-1,e)$ and $b=(m-1)/d$. The eigenvalue set of $y$ is $\La \cup \La_1 \cup \cdots \cup \La_d$,  where $\La = \{\l_1,\l_1^{q_0}\}$ and $\La_i = \{\l_2^{q_0^i}, \l_2^{-q_0^i}, \dots, (\l_2^{q_0^i})^{q^{b-1}}, (\l_2^{q_0^i})^{-q^{b-1}} \}$, for each $i$. Since the eigenvalue sets of $M$ and $N$ are closed under the map $\a \mapsto \a^q$, the eigenvalue set of $M$ is $\La \cup \La_{a_1} \cup \cdots \cup \La_{a_l}$ and the eigenvalue set of $N$ is $\La_{a_{l+1}} \cup \cdots \cup \La_{a_d}$ where $l \geq 1$ and $\{a_1, \dots, a_d \} = \{1, \dots, d\}$. Therefore, $b$ divides $m$ and $m-2$. Thus, $b$ divides $2$, so $\La_i = \{ \l_2^{q_0^i}, \l_2^{-q_0^i}\}$, for each $i$. In particular, $d=m-1$. \par

By arguing as in Case 1, we can show that $|C_G(\oly)| = |C_H(\oly)|$ and that $\oly^G \cap H$ splits into $\binom{m}{\frac{m}{2}}$ $B$-classes (corresponding to choosing $m/2$ of $\La, \La_1, \dots, \La_{m-1}$ for $M$) which fuse to $\frac{1}{2}\binom{m}{\frac{m}{2}}$ $H$-classes. So $\oly$, and thus $t\th$, lies in at most $\frac{1}{2}\binom{m}{\frac{m}{2}}$ $G$-conjugates of $H$. \par \vspace{2pt}

\noindent\textbf{Case 3: $k = 2$ and $V_1,V_2$ are totally isotropic}

Assume that $m$ is odd. Then $y = [A_2, A_{2m-2}]$, and, since $\oly \in B$, $y = [M,M^{-T}]$ for $M \in \GL_m(q)$. By Lemma~\ref{lem:ABCD}, $y$ has eigenvalue set $\{ \l_1, \l_1^{q_0},\l_2, \l_2^{q_0}, \dots,\l_2^{q_0^{2m-2}} \}$. Since $\l_1^{q_0} = \l_1^{-1}$, assume that $\l_1$ is an eigenvalue of $M$ and $\l_1^{q_0}$ is an eigenvalue of $M^{-T}$. \par

Let $d = (2m-2,e)$ and $b=(2m-2)/d$. The eigenvalue set of $y$ is $\La \cup \La_1 \cup \cdots \cup \La_d$, where $\La = \{\l_1,\l_1^{q_0}\}$ and where $\La_0, \dots, \La_d$ are the orbits of the eigenvalue set of $A_{2m-2}$ under the map $\a \mapsto \a^q$. Since the eigenvalue set of $M$ is closed under the map $\a \mapsto \a^q$, the eigenvalue set of $M$ is $\{\l_1\} \cup \La_{a_1} \cup \cdots \cup \La_{a_l}$ where $l = \frac{d}{2}$ and where $a_1, \dots, a_l \in \{1, \dots, d\}$ are distinct. If $b$ is even, then $\La_i^{-1} = \La_i$, for each $i$. However, this contradicts the distinctness of the eigenvalues of $y$. Therefore, $b$ is odd. \par

As in Case 1, we can show that $|C_G(\oly)| = |C_H(\oly)|$.  Additionally, if $N \in \GL_n(q)$ has eigenvalue set $\{ \l_1^{\e} \} \cup \La_1^{\e_1} \cup \cdots \cup \La_l^{\e_l}$, then a $G$-conjugate of $y$ is $B$-conjugate to $[N,N^{-T}]$ for exactly one choice of $(\e,\e_1,\dots,\e_l) \in \{+,-\}^{l+1}$. Therefore, $y^G$ splits into $2^{l+1}$ $B$-classes, which fuse to $2^l$ $H$-classes. So $y$, and thus $t\th$, lies in at most $2^l = 2^{(2m-2,e)/2} \leq 2^{(m-1,e)}$ $G$-conjugates of $H$. When $m$ is even, the analysis is very similar and we omit the details. \par \vspace{4pt}

We have now established Proposition~\ref{prop:MaximalSubgroups} in cases $\mathbf{S}$ and $\mathbf{O}$. For $\mathbf{S_4}$ the argument is similar but briefer. Let $T=\PSp_4(q)$ and assume that $\th$ is not a graph-field automorphism. The possible types of irreducible imprimitive subgroups are $\Sp_2(q) \wr S_2$ and $\GL_2(q).2$. In order to prove Proposition~\ref{prop:MaximalSubgroups}, we need to show that if $t\th$ is contained in a subgroup of type $\Sp_2(q) \wr S_2$, then $e$ is even and $t\th$ is contained in a unique such subgroup. \par

Suppose that $t\th$ preserves a decomposition $V=V_1 \oplus V_2$ where $V_1$ and $V_2$ are non-degenerate 2-spaces. Let $H$ be the stabiliser in $G$ of this decomposition, and let $B$ be the index two subgroup of $H$ stabilising each summand. Recall that a suitable power of $t\th$ lifts to an $X$-conjugate of an element $g \in \Sp_{4}(q_0)$ which has distinct eigenvalues and whose order is a ppd of $q_0^4-1$. (Note that $g$ is $y$, $y^{\ell}$ or $y^{(q_0-1)\ell}$, depending on $\th$; see Table~\ref{tab:Elements}.) Since $t\th$ preserves the direct sum decomposition so does $g$. However, $g$ has odd order, so $g$ stabilises each summand. Therefore, each of the eigenvalues of $g$ is contained in $\F_{q^2}$. In particular, since $q=q_0^e$ and the eigenvalues of $g$ are not contained in a proper subfield of $\F_{q_0^4}$, it must be that $e$ is even. Now, for some $M,N \in \Sp_2(q_0)$, $g = [M, N]$, and \[ |C_G(g)| = |\GL^{\e}_1(q)||\GL^{\e}_1(q)| = |C_B(g)| = |C_H(g)|. \] Moreover, as in Case 1 above, $g^G \cap H = g^H$. Therefore, $g$, and thus $t\th$, lies in at most one $G$-conjugate of $H$. Together with Corollary~\ref{cor:MultiplicityCentraliser}, this completes the proof. \qedhere
\end{proof}

Before proving Proposition~\ref{prop:MaximalSubgroups} for primitive subgroups, we will present further results on the multiplicities of subgroups. The first of these results, which pertains to subfield subgroups, is a generalisation of \cite[Prop. 2.16(ii)]{ref:BurnessGuest13} and the proof is very similar.
\begin{proposition}\label{prop:MultiplicitiesSubfield}
Let $g\s \in G$ be such that $f(g\s)$ lifts to $[M_1,\dots,M_k]$ where for each $i$, $M_i = A_{d_i}$, $M_i = B_{2d_i}$ or $M_i = I_1$, and $d_1,\dots,d_k$ are distinct. Let $H$ be a maximal subfield subgroup of $G$ over the field $\F_{q_0}$. Then $g\s$ is contained in at most $e^k$ $G_1$-conjugates of $H$. 
\end{proposition}

\begin{corollary}\label{cor:MultiplicitySubfield}
Suppose that $e$ is prime and let $H$ be a maximal subfield subgroup of $G$ over the field $\F_{q_0}$. Let $t\th \in G$ be the element defined in Table~\ref{tab:Elements}. Then there are at most $e^k$ subgroups of type $H$ in $\M(G,t\th)$ where $k=1$ in $\mathbf{S_4}$, $k=2$ in $\mathbf{S}$ and $k=3$ in $\mathbf{O}$.
\end{corollary}

\begin{proof}
By Proposition~\ref{prop:IsomorphismIsConjugacy}, all maximal subfield subgroups over $\F_{q_0}$ are $G_1$-conjugate. The result now follows from Proposition~\ref{prop:MultiplicitiesSubfield} and the choice of element $f(t\th)$ in Table~\ref{tab:Elements}.
\end{proof}

The following is an application of Proposition~\ref{prop:MultiplicityOrthogonal}.
\begin{corollary}\label{cor:MultiplicityOrthogonal}
Let $q$ be even, $T=\Sp_{2m}(q)$ and $\th$ a field automorphism. Then the element $t\th \in G$, defined in Table~\ref{tab:Elements}, is contained in exactly one subgroup of $G$ of type $\O_{2m}^+(q)$ or $\O_{2m}^-(q)$.
\end{corollary}

\begin{proof}
By Proposition~\ref{prop:MultiplicityOrthogonal}, it suffices to show that $y$ is contained in exactly one subgroup of $\Sp_{2m}(q_0)$ of type $\O_{2m}^{\e}(q_0)$. If $m \geq 3$ is odd, then $y$ has order $(q_0+1)(q_0^{m-1}+1)$ and, hence, is not contained in a subgroup of type $\O^{-}_{2m}(q_0)$. Since the $\Sp_{2m}(q_0)$- and $\O_{2m}^{+}(q_0)$-conjugacy of semisimple elements of odd order is determined by eigenvalues, $y^G \cap H = y^H$. Moreover, \[ |C_{\Sp_{2m}(q_0)}(y)|=(q_0+1)(q_0^{m-1}+1)=|C_{\O_{2m}^+(q_0)}(y)|\] (see \cite[Appendix B]{ref:BurnessGiudici16}). Thus, by Lemma~\ref{lem:MultiplicitiesGeneralFixed}, $y$ is contained in exactly one subgroup of type $\O^-_{2m}(q_0)$. Similar arguments apply when $m \geq 2$ is even, and we omit the details.
\end{proof}

We are now in a position to complete the proof of Proposition~\ref{prop:MaximalSubgroups} in cases $\mathbf{S}$ and $\mathbf{O}$.
\begin{proposition}\label{prop:MaximalSubgroupsPrimitive}
In cases $\mathbf{S}$ and $\mathbf{O}$, Proposition~\ref{prop:MaximalSubgroups} is true for irreducible primitive subgroups.
\end{proposition}

\begin{proof}
The stated upper bounds on the multiplicities follow from Corollaries~\ref{cor:MultiplicityCentraliser}, \ref{cor:MultiplicitySubfield} and \ref{cor:MultiplicityOrthogonal}. Therefore, we will focus on determining the types of subgroups which arise. Let $H \in \M(G,t\th)$ be irreducible and primitive. By \cite[Main Theorem]{ref:KleidmanLiebeck}, $H$ lies in one of the geometric families $\C_3, \dots, \C_8$ or is an almost simple irreducible group in the $\S$ collection. By construction, a power of $t\th$ is $X$-conjugate to $\oly$. Moreover, by the choice of $\oly$, a suitable power of $t\th$ is $X$-conjugate to an element $z$ of odd prime order which lifts to a matrix $[M, I_{n-2}]$, where $M \in \GL_2(q_0)$. \par

Consider $\C_3$ subgroups. By \cite[Lemma 4.2]{ref:LiebeckShalev99}, if $g \in G$ is contained in a field extension subgroup of degree $k$, then $\nu(g) \geq k$ (see Notation~\ref{not:nu}). However, $\nu(z)=2$, so $k=2$. Hence, if $H \in \C_3$, then $T=\PSp_{2m}(q)$ and either $H$ has type $\Sp_m(q^2)$, or $q$ is odd and $H$ has type $\GU_m(q)$. We will show that neither of these possibilities occur. \par

First consider subgroups of type $\Sp_m(q^2)$. A preimage of $z=[\l,\l^{q_0},I_{2m-2}]$ in $\Sp_{m}(q^2)$ has exactly one non-trivial eigenvalue, by \cite[Lemma 5.3.11]{ref:BurnessGiudici16}. However, this is impossible, so $z$ is not contained in a subgroup of type $\Sp_m(q^2)$. Now consider subgroups of type $\GU_m(q)$. If $m$ is even, then over $\F_{q_0}$ a power of $y$ is $[I_2,B_{2m-2}]$, which, by \cite[Lemma 5.3.2]{ref:BurnessGiudici16}, is not contained in a subgroup of type $\GU_m(q)$ since $m-1$ is odd. If $e$ is even, then over $\F_{q_0}$ a power of $y$ is $[A_2,I_{2m-2}]$, which over $\F_{q}$ has the form $[B_2,I_{2m-2}]$, which, as above, is not contained in $H$. Finally, if $m$ and $e$ are odd, then over $\F_{q_0}$ a power of $y$ is $[I_2,A_{2m-2}]$, which over $\F_{q}$ has the form $[I_2,A_{2d_1}, \dots, A_{2d_k}]$, where $d_i$ is even since $m-1$ is even and $e$ is odd. By \cite[Lemma 5.3.2]{ref:BurnessGiudici16}, this element is not contained $H$ since $d_i$ is even. \par

Now let us turn to $\C_4$ subgroups. By \cite[Lemma 3.7]{ref:LiebeckShalev99}, if $g$ has prime order and preserves a tensor product decomposition $V=U_1 \otimes U_2$, then $\nu(g) \geq \max\{\dim U_1, \dim U_2\}$. However, $\nu(z) = 2$, so $\dim U_1, \dim U_2 \leq 2$. Hence, $n \leq 4$: a contradiction. Therefore, $H \not\in \C_4$. \par

Write $T = \Sigma_n(q)$ where $\Sigma \in \{ \PSp, \Om \}$. If $H \in \C_5$, then $H$ has type $\Sigma_n(q_1)$ with $q=q_1^l$ for a prime $l$. Since $f(t\th)$ has order divisible by a ppd of $q_0^{(2m-2)/(m,2)}-1$, $f(t\th) \not\in \Sigma_{n}(F)$ for any proper subfield $F$ of $\F_{\!q_0}$. However, $f(t\th) \in \Sigma_n(q_1) \cap \Sigma_n(q_0) \leq \Sigma_n(\F_{q_0} \cap \F_{q_1})$, so $\F_{\!q_0} \cap \F_{\!q_1} = \F_{\!q_0}$. That is, $\F_{\!q_0} \leq \F_{\!q_1} \leq \F_q$. So $q_1=q_0^d$ for some $d$. \par

Since $q$ is not prime, $H \not\in \C_6$. \par

We treat $\C_7$ similarly to $\C_4$. By \cite[Lemma 7.1]{ref:Burness072}, if $g$ has prime order and preserves $V=U_1 \otimes \cdots \otimes U_t$ where $\dim U_1 = \cdots = \dim U_t = a$, then $\nu(g) \geq a^{t/2}$. However, $\nu(z) = 2$, so $a =1$ or $(a,t)=(2,2)$. Hence, $n \leq 4$: a contradiction. Therefore, $H \not\in \C_7$. \par

If $H \in \C_8$ then $T=\Sp_{2m}(q)$, $q$ is even and $H$ has type $\O^{\e}_{2m}(q)$ for $\e \in \{+,-\}$. \par

It remains to consider the $\S$ family. By \cite[Theorem 7.1]{ref:GuralnickSaxl03}, since $n \geq 6$, if $H \in \S$, then $\nu(g) > 2$ for all $g \in H$ or $H$ belongs to a known list of exceptions (see \cite[Table 2.3]{ref:Burness074}, for a convenient list of the exceptions). Since $q$ is not prime, $H$ is not an alternating or symmetric group acting on the fully deleted permutation module. Therefore, since $\nu(z) = 2$, the possibilities are
\begin{enumerate}[(i)]
\item{$T=\PSp_6(q)$ and $q$ odd: $H=J_2$;}
\item{$T=\Sp_{6}(q)$ ($q$ even) or $T=\Om_7(q)$: $H=G_2(q)'$.}
\end{enumerate}
First consider case (i). The order of $y$ is $l=\mathrm{lcm}(q_0+1,q_0^2+1)$. If $q_0 \equiv 1 \mod{4}$, then $y^{l/2}=-I_6$ and $\oly$ has order $l/2 \geq \mathrm{lcm}(5+1,5^2+1)/2 =39$. Otherwise, $\oly$ has order $l \geq \mathrm{lcm}(3+1,3^2+1) = 20$. In either case, $\oly$ is not contained in a subgroup of type $J_2$ since the maximum order of an element of $J_2$ is 15 (see \cite{ref:ATLAS}). \par

Now consider case (ii). Assume that $T=\Om_7(q)$; a very similar argument applies to $\Sp_6(q)$. A suitable power of $t\th$ is an $X$-conjugate of $g = [\l_1, \l_1^{q_0}, \l_2, \l_2^{q_0}, \l_2^{q_0^2}, \l_2^{q_0^3}, 1]$ where $\l_1 \in \F_{q_0^2}$ has order $q_0+1$ and $\l_2 \in \F_{q_0^4}$ has order $q_0^2+1$. (Either $g=y$ or $g=y^{q-1}$, depending on $\th$; see Table~\ref{tab:Elements}.)  It is well-known that $\SL_3(q) \leq G_2(q)$ and that the restriction of $V$ to $\SL_3(q)$ is $U \oplus U^* \oplus 0$ where $U$ is the natural $\SL_3(q)$ module and $0$ is the trivial module. Therefore, $g = [\a_1, \a_2, \a_3, \a_1^{-1}, \a_2^{-1}, \a_3^{-1},1]$. By the orders of the eigenvalues, without loss of generality, let $\a_1 = \l_1$ and $\a_2 = \l_2$. Since $\l_2^{-1} = \l_2^{q_0^2}$ we must have either: (a) $\a_3 = \l_2^{q_0}$ or (b) $\a_3 = \l_2^{q_0^3}$. Since $[\a_1,\a_2,\a_3] \in \SL_3(q)$, $\a_1\a_2\a_3=1$. If (a) holds, then \[ \l_1^{q_0} = \l_1^{-1} = \a_1^{-1} = \a_2\a_3 = \l_2\l_2^{q_0} = \l_2^{1+q_0} \] \[ \l_1 = \a_1 = (\a_2\a_3)^{-1} = \l_2^{-1}\l_2^{-q_0} = \l_2^{q_0^2}\l_2^{q_0^3} = \l_2^{q_0^2+q_0^3}. \]  Therefore, $\l_1^{q_0^2} = (\l_1^{q_0})^{q_0} = \l_2^{q_0+q_0^2}$. Since $\l_1 \in \F_{q_0^2}$, $\l_1 = \l_1^{q_0^2}$, so \[ \l_1^{q_0} = (\l_1^{q_0^2})^{q_0} = (\l_2^{q_0+q_0^2})^{q_0} = \l_2^{q_0^2+q_0^3} = \l_1.\] Therefore, $\l_1 \in \F_{q_0}$: a contradiction. Case (b) is similar. Hence, $H$ is not $G_2(q)$. This completes the proof.
\end{proof}

For the case $\mathbf{S_4}$ we need two further results on subgroup multiplicities.
\begin{proposition}\label{prop:MultiplicitySuzuki}
Let $q$ be even, $T=\Sp_4(q)$ and $\th$ an involutory graph-field automorphism. Then the element $t\th \in G$, defined in Table~\ref{tab:Elements}, is contained in exactly one subgroup of $G$ of type $Sz(q)$.
\end{proposition}

\begin{proof}
By \cite[Table 8.14]{ref:BrayHoltRoneyDougal}, there is a unique $G$-class of subgroups of type $Sz(q)$. Let $H = C_G(\th) = C_T(\th) \times \< \th \> \cong Sz(q) \times \< \th \>$. We need to show that $t\th$ is contained in exactly one $G$-conjugate of $H$. Thus, if we assume that $t\th \in H$, by Lemma~\ref{lem:MultiplicitiesGeneralFixed}, it suffices to show that $|C_G(t\th)| = |C_H(t\th)|$ and $(t\th)^G \cap H = (t\th)^H$. \par

Let us first show that $|C_G(t\th)| = |C_H(t\th)|$. By Proposition~\ref{prop:SympShintaniEvenGraphField}, the Shintani map \[ f\: \{ (g\th)^T \mid t \in T \} \to \{ x^{Sz(q)} \mid x \in Sz(q) \} \] is defined as $f(g\th) = a^{-1}(g\th)^2a$ where $a^{-\th^{-1}}a=g$. By Theorem~\ref{thm:ShintaniDescent}(ii), \[ |C_G(t\th)| = 2|C_T(t\th)| = 2|C_{Sz(q)}(f(t\th))|. \] By construction, $f(t\th) \in Sz(q)$ has order a ppd $r$ of $q^4-1$. Since $r$ divides $q + \e \sqrt{2q} + 1$ for some $\e \in \{1,-1\}$, by \cite[Prop. 16]{ref:Suzuki62}, \[|C_{Sz(q)}(x)|=q + \e \sqrt{2q} + 1,\] for every element $x \in Sz(q)$ of order $r$. Since $t\th$ has order $2r$, $t$ has order $r$ and \[ 2|C_{Sz(q)}(f(t\th))| = 2(q + \e \sqrt{2q} + 1) = 2|C_{C_{T}(\th)}(t\th)| = |C_{H}(t\th)|. \]

We will now prove that $(t\th)^G \cap H = (t\th)^H$. Let $s\th \in H$ be $G$-conjugate to $t\th$. We will first show that $s$ and $t$ are $T$-conjugate. By Remark~\ref{rem:ShintaniDescent}(ii), $s\th$ and $t\th$ are $T$-conjugate. Therefore, $s^2=(s\th)^2$ and $t^2 = (t\th)^2$ are $T$-conjugate. Record that $s, t \in C_T(\th) \leq T$ have order $r$. Since $r$ is odd, the square map on $T$ permutes the $T$-classes of order $r$. Therefore, since $s^2$ and $t^2$ are $T$-conjugate, $s$ and $t$ are $T$-conjugate. \par

We will now verify that $s\th$ and $t\th$ are $C_T(\th)$-conjugate. Observe that it suffices to show that $s$ and $t$ are $C_T(\th)$-conjugate. Since $s$ and $t$ are $T$-conjugate it suffices to show that no two $C_{T}(\th)$-classes of elements of order $r$ are fused into one $T$-class. Since $r$ does not divide $|T:C_{T}(\th)| = q^2(q-1)(q^2-1)$, every element of $T$ of order $r$ is $T$-conjugate to an element of $C_T(\th)$. Hence, it suffices to verify that there are the same number of classes of elements of order $r$ in $C_{T}(\th) \cong Sz(q)$ and $T \cong \Sp_4(q)$. \par

First consider $Sz(q)$. Let $\A$ be the set of centralisers of elements of order $r$ in $Sz(q)$. By \cite[Prop. 16]{ref:Suzuki62}, for all $A \in \A$, $|A| = q + \e\sqrt{2q} + 1$ and $C_{Sz(q)}(A)=A$. In particular, two members of $\A$ are either equal or intersect trivially. Moreover, by \cite[Theorem 9]{ref:Suzuki62}, all members of $\A$ are $Sz(q)$-conjugate. Since $|N_{Sz(q)}(A)| = 4|A|$, for all $x \in A$, $|x^{Sz(q)} \cap A| = 4$. Therefore, there are $(r-1)/4$ conjugacy classes of elements of order $r$ in $Sz(q)$. Now consider $\Sp_4(q)$. The conjugacy classes of elements of order $r$ in $\Sp_4(q)$ are represented by the elements $[ \l, \l^q, \l^{q^2}, \l^{q^3} ]$ where $\l \in \F_{q^4}$ is a non-trivial $r^{th}$ root of unity. So there are $(r-1)/4$ conjugacy classes of elements of order $r$. This establishes that $(t\th)^G \cap H = (t\th)^H$ and, thus, proves the result.
\end{proof}

\begin{proposition}\label{prop:MultiplicityFieldExtension}
Let $T=\PSp_4(q)$ and assume that $\th$ is not a graph-field automorphism. Let $t\th \in G$ be the element defined in Table~\ref{tab:Elements}. If $t\th$ is contained in a subgroup of $G$ of type $\Sp_2(q^2)$, then $e$ is odd and $t\th$ is contained in at most one such subgroup.
\end{proposition}

\begin{proof}
Let $H \leq G$ have type $\Sp_2(q^2)$. Write $L=G \cap \PGL(V)$ and let $H_0 = H \cap L = B\sp\<\psi\>$ where $B=\PSp_2(q^2)$ and $\psi$ is an involutory field automorphism of $B$. Suppose that $t\th \in H$. By construction, an $X$-conjugate of a power of $t\th$ lifts to a prime order element $x=[\l,\l^{q_0},\l^{q_0^2},\l^{q_0^3}]$ where $\l \in \F_{q_0^4}$ is not contained in a proper subfield of $\F_{q_0^4}$. \par

First suppose that $e$ is even. For $\mu = \l^{q_0}$, $x=[\l,\l^q, \mu, \mu^q]$ if $e \equiv 2 \mod{4}$, and $x=[\l,\l^{-1}, \mu, \mu^{-1}]$ if $e \equiv 0 \mod{4}$. Since $B$ embeds in $L$, modulo scalars, as $[\l_1,\l_2] \mapsto [\l_1,\l_1^q,\l_2,\l_2^q]$, neither of these possibilities for $x$ are images of elements of $B$. \par

Now suppose that $e$ is odd. Then $x=[\l,\l^q,\l^{q^2},\l^{q^3}]$. Let $h$ be a preimage of $x$ in $H$. Then $h \in B$ and $h$ lifts to either $[\l,\l^{q^2}]$ or $[\l^q,\l^{q^3}]$. However, $[\l,\l^{q^2}]$ and $[\l^q,\l^{q^3}]$ are $H_0$-conjugate (although not $B$-conjugate). Therefore, $|x^L \cap H_0| = |x^{H_0}|$. Moreover, $|C_L(x)| = q^2+1 = |C_{H_0}(x)|$ (see \cite[Appendix B]{ref:BurnessGiudici16}, for example). Hence, by Lemma~\ref{lem:MultiplicitiesGeneralFixed}, $x$, and hence $t\th$, is contained in at most one $G$-conjugate of $H$. This completes the proof.
\end{proof}

\begin{proposition}\label{prop:MaximalSubgroupsPrimitiveS4}
In case $\mathbf{S_4}$, Proposition~\ref{prop:MaximalSubgroups} is true for irreducible primitive subgroups.
\end{proposition}

\begin{proof}
By \cite[Tables 8.12--8.14]{ref:BrayHoltRoneyDougal}, since $q \neq p$, the only types of irreducible primitive maximal subgroups which arise are those given in Table~\ref{tab:MaximalSubgroups}. The uniqueness of the subgroups of type $\O_4^{\e}(q)$ and $\Sp_2(q^2)$, when they occur, follows from Corollary~\ref{cor:MultiplicityOrthogonal} and Proposition~\ref{prop:MultiplicityFieldExtension}. If $q$ is even, $\th$ is a graph-field automorphism and $e=1$, then the uniqueness of the subgroup of type $Sz(q)$ follows from Proposition~\ref{prop:MultiplicitySuzuki}. Moreover, in this case, no subgroups of type $\O^{\e}_{2}(q) \wr S_2$ occur since the order of $t\th$ is divisible by a ppd of $q^2+1$, which does not divide the order of these groups. The remaining multiplicities follow by Corollaries~\ref{cor:MultiplicityCentraliser} and \ref{cor:MultiplicitySubfield}.
\end{proof}

We have now proved Proposition~\ref{prop:MaximalSubgroups}.

\subsection{Probabilistic method}\label{ssec:ProofProbMethod}
Let $G = \<T,\th\> \in \A$ with $T \neq \PSp_4(2)'$ and $\th \not\in \InnDiag(T)$. Maintain the notation introduced at the opening of Section~\ref{sec:Proof}. Fix $t\th$ as the element defined in Table~\ref{tab:Elements}. In this section, we will use probabilistic techniques to establish our main results on uniform spread. (Some asymptotic results will be proved in Section~\ref{ssec:ProofAsymptotic}.) \par

Let us begin by recalling the definition \[ P(x,s) = 1 - \frac{|\{z \in s^G \mid G = \< x,z \>\}|}{|s^G|}. \] We can now state the key lemma, which encapsulates our probabilistic method.
\begin{lemma}\label{lem:ProbMethod}
Let $G$ be a finite group and let $s \in G$.
\begin{enumerate}
\item{For $x \in G$, \[ P(x,s) \leq \sum_{H \in \M(G,s)}^{} \fpr(x,G/H).\]}
\item{If for all $k$-tuples $(x_1,\dots,x_k)$ of prime order elements of $G$ \[ \sum_{i=1}^{k}P(x_i,s) < 1,\] then $G$ has uniform spread $k$ with respect to the conjugacy class $s^G$.}
\end{enumerate}
\end{lemma}

\begin{proof}
See \cite[Lemmas 2.1 \& 2.2]{ref:BurnessGuest13}.
\end{proof}

Let us introduce a piece of notation. For an integer $k$, define
\[
\pi_k =
\left\{ 
\begin{array}{ll}
1 & \text{if $k$ is even} \\
0 & \text{if $k$ is odd}  \\
\end{array}
\right.
\]

We will now consider the cases $\mathbf{S}$, $\mathbf{O}$ and $\mathbf{S_4}$ in turn.
\begin{proposition}\label{prop:ProbSymp}
Let $m \geq 3$ and $G = \< \PSp_{2m}(q), \th \>$ where $\th \in \Aut(\PSp_{2m}(q))$.
\begin{enumerate}[itemsep=3pt]
\item{If $q$ is even, then $u(G) \geq 2$.}
\item{If $q$ is odd, then $u(G) \geq 4$.}
\item{As $q \to \infty$, $u(G) \to \infty$.}
\item{If $m \geq 16$, then $u(G) \geq q-1$.}
\end{enumerate}
\end{proposition}

\begin{proof}
We will apply Lemma~\ref{lem:ProbMethod} with $s=t\th$. Let $x \in G$ have prime order. Proposition~\ref{prop:MaximalSubgroups} gives a superset of $\M(G,t\th)$ and together with the fixed point ratios in Propositions~\ref{prop:FPRs}, \ref{prop:FPRsSubspace} and \ref{prop:FPRs2Space} we obtain 
\begin{align*}
P(x,t\th) < \left( \frac{1}{q^2} + \frac{1}{q^4} + \frac{2}{q^{2m-2}} + \frac{1}{q^{2m-1}} \right) &+ \pi_m \left( \frac{6}{q^{m-1}} + \frac{2}{q^m} \right) \\
&+ N \frac{(2q+2)^{1/2}}{q^{m-1}} + \pi_q \left( \frac{1}{q} + \frac{1}{q^m-1} \right),
\end{align*}
where \[ N = 1 + N_{nd} \cdot q + N_{ti} + N_{s} \] and $N_{nd}$, $N_{ti}$ and $N_{s}$ are the numbers of subgroups in $\M(G,t\th)$ of type $\Sp_m(q) \wr S_2$, $\GL_m(q).2$ and subfield subgroups, respectively. (The factor of $q$ associated with $N_{nd}$ is to account for the fact that, in this case, $\ell = 2$; see Proposition~\ref{prop:FPRs}.) \par

From Proposition~\ref{prop:MaximalSubgroups}, Corollary~\ref{cor:MultiplicityCentraliser} and the fact that $e$ has at most $2 + \log{\log{q}}$ distinct prime divisors,
\begin{equation}
N \leq 1 + (\pi_m \cdot q + \pi_{q+1} + (2 + \log\log{q}))(q^{m/2}+q^{(m-1)/2}+q^{1/2}+1). \label{eq:N5}
\end{equation} 
This yields
\begin{align*}
P(x,t\th) &< \frac{1}{q} + \frac{1}{q^{2}} + \frac{1}{q^{m/2-5}} + \left(\frac{1}{q^{m-5}} + \frac{2}{q^{m-3}} + \frac{6}{q^{m-1}} + \frac{1}{q^m-1} + \frac{2}{q^m} + \frac{3}{q^{2m-2}} \right).
\end{align*}
Therefore, as $q \to \infty$, $P(x,t\th) \to 0$ and, consequently, $u(G) \to \infty$. Moreover, if $m \geq 16$, then \[ P(x,t\th) \leq \frac{1}{q} + \frac{1}{q^2} + \frac{1}{q^3} + \frac{1}{q^4} < \frac{1}{q-1}\] and, consequently, $u(G) \geq q-1$. This proves (iii) and (iv). \par

Let us now prove (i) and (ii). We will consider various cases depending on $e$, $m$ and $q$. First suppose that $e \geq 5$. The upper bound in \eqref{eq:N5} decreases with $q$ and $m$. Therefore, considering $m=3$ and $q=2^5$ shows that $P(x,\th) < 1/2$ for $q$ even, and considering $m=3$ and $q=3^5$ shows that $P(x,t\th) < 1/4$ for $q$ odd. \par

Now suppose that $e=4$. Since $e$ has a unique prime divisor, by Proposition~\ref{prop:MaximalSubgroups},
\begin{equation*}
N \leq 1 + \pi_m \cdot \frac{1}{2}\binom{m}{\frac{m}{2}} \cdot q + \pi_{q+1} \cdot 2^{(m-1,e)} + (q^{m/4}+q^{(m-1)/4}+q^{1/4}+1)
\end{equation*} 
and with this bound the result can be verified. \par

Finally suppose that $e \in \{2, 3\}$. Since $e$ is prime, by Proposition~\ref{prop:MaximalSubgroups},
\begin{equation*}
N \leq 1 + \pi_m \cdot \frac{1}{2}\binom{m}{\frac{m}{2}} \cdot q + \pi_{q+1} \cdot 2^{(m-1,e)} + e^2.
\end{equation*} 
With this bound, the result follows unless $(m,q) \in \{ (3,4), (4,4), (3,8), (3,9) \}$. \par

Let $(m,q)=(3,9)$. Since $\frac{2m-2}{(2m-2,e)} = 2$ is even, $N_{ti}=0$ (see Table~\ref{tab:MaximalSubgroups}). Together with the refined bound from Proposition~\ref{prop:FPRs} (Table~\ref{tab:FPRsExtra}), we can verify that $P(x,t\th) < \frac{1}{4}$. \par

Now let $(m,q) \in \{ (3,8), (4,4) \}$. If $x \not\in \PGL(V)$ or $\nu(x) > 1$, then we have improved bounds for the $\C_1$ and $\C_8$ subgroups from Propositions~\ref{prop:FPRsSubspace} and \ref{prop:FPRs2Space} and we may use the refined bound in Proposition~\ref{prop:FPRs}. If $x \in \PGL(V)$ and $\nu(x)=1$, then we have specialised bounds for the subfield subgroups from Proposition~\ref{prop:FPRsTrans}. In both cases, $P(x,t\th) < \frac{1}{2}$. \par

Finally, let $(m,q) = (3,4)$. Arguing as above, if $x \not\in \PGL(V)$ or $\nu(x) > 1$, then $P(x,t\th) < 0.254$, and if $x \in \PGL(V)$ and $\nu(x) = 1$, then $P(x,t\th) < 0.601 < 1 - 0.254$. Therefore, for all $x_1, x_2 \in G$ of prime order there exists $g \in G$ such that $\< x_1, (t\th)^g \> = \< x_2, (t\th)^g \> = G$ unless $x_1, x_2 \in \PGL(V)$ and $\nu(x_1)=\nu(x_2)=1$. In this case, we can verify in \textsc{Magma} that there exists $g \in G$ such that $\< x_1, (t\th)^g \> = \< x_2, (t\th)^g \> = G$.
\end{proof}

\begin{proposition}\label{prop:ProbOrth}
Let $q$ be odd, $m \geq 3$ and $G = \< \Om_{2m+1}(q), \th \>$ where $\th \in \Aut(\Om_{2m+1}(q))$.
\begin{enumerate}[itemsep=3pt]
\item{For all $G$, $u(G) \geq 3$.}
\item{As $q \to \infty$, $u(G) \to \infty$.}
\item{If $m \geq 18$, then $u(G) \geq q-1$.}
\end{enumerate}
\end{proposition}

\begin{proof}
Let $x \in G$ have prime order. Proposition~\ref{prop:MaximalSubgroups} gives a superset of $\M(G,t\th)$ and together with the fixed point ratios in Propositions~\ref{prop:FPRs}--\ref{prop:FPRsSubspace} we obtain \[ P(x,t\th) < \frac{1}{q} + \frac{1}{q^2} + \frac{1}{q^3} + \frac{1}{q^{m-2}} + \frac{14}{q^{m-1}} + \frac{5}{q^m} + \frac{2}{q^{(m+1)/2}} + N\frac{(4q+4)^{1/2}}{q^{m-1}}  \] where $\M(G,t\th)$ contains $N$ subfield subgroups. Since $e$ has at most $2 + \log{\log{q}}$ distinct prime divisors,
\begin{align*}
N &\leq (2 + \log\log{q})(q^{(m+1)/2}+q^{m/2}+q+q^{1/2}).
\end{align*}
Therefore,
\begin{align*}
P(x,t\th) &< \frac{1}{q} + \frac{1}{q^2} + \frac{1}{q^3} + \frac{1}{q^{m/2-5}} + \frac{2}{q^{(m+1)/2}} + \frac{1}{q^{m-5}} + \frac{1}{q^{m-2}} + \frac{14}{q^{m-1}} + \frac{5}{q^m},
\end{align*}
so $P(x,t\th) \to 0$, and $u(G) \to \infty$, as $q \to \infty$ . Moreover, if $m \geq 18$, then $P(x,t\th) < \frac{1}{q-1}$, and $u(G) \geq q-1$. Finally, unless $\soc(G) = \Om_7(9)$, it is straightforward to show that $P(x,t\th) < \frac{1}{3}$, and $u(G) \geq 3$, by arguing as in the proof of Proposition~\ref{prop:ProbSymp}. In the case that $\soc(G)=\Om_7(9)$ we apply the same approach, but for the subspace subgroups we determine the fixed point ratios using $\textsc{Magma}$.
\end{proof}

\begin{proposition}\label{prop:ProbS4}
Let $G = \< \PSp_{4}(q), \th \>$ where $\th \in \Aut(\PSp_{4}(q))$.
\begin{enumerate}[itemsep=3pt]
\item{For all $G$, $u(G) \geq 2$.}
\item{As $q \to \infty$, $u(G) \to \infty$.}
\item{If $\th$ is an involutory graph-field automorphism, then $u(G) \geq q^2/18$.}
\end{enumerate}
\end{proposition}

\begin{proof}
For $q \in \{4,8,9,16,25,27\}$ the result can be verified computationally in \textsc{Magma} (see Table~\ref{tab:Comp}). Therefore, suppose that $q \geq 32$. Let $x \in G$ have prime order. \par

First suppose that $\th$ is a field automorphism. Proposition~\ref{prop:MaximalSubgroups} gives a superset of $\M(G,t\th)$ and together with the fixed point ratios in Propositions~\ref{prop:FPRsSubspace} and \ref{prop:FPRsSymp4} we obtain 
\begin{align}
P(x,t\th) &\leq \frac{4(q_0^2 + 1)(3 + 2 + \log\log{q})}{q(q-1)} + \frac{q}{q^2-1} + \frac{1}{q} + \frac{1}{q^2-1} \label{eq:q0bound} \\
&\leq \frac{4(q+1)(3 + 2 + \log\log{q})}{q(q-1)} + \frac{q}{q^2-1} + \frac{1}{q} + \frac{1}{q^2-1}. \label{eq:qbound}
\end{align}
The asymptotic statement in (ii) now follows from \eqref{eq:qbound}. If $q \geq 64$, then $P(x,t\th) < \frac{1}{2}$ by \eqref{eq:q0bound}. If $q=32$, then $q_0 = 2$ and $P(x,t\th) < \frac{1}{2}$, by \eqref{eq:q0bound}. Therefore, $u(G) \geq 2$. \par

Now suppose that $\th$ is a graph-field automorphism. Therefore, $q$ is even and has a unique prime divisor. By Propositions~\ref{prop:FPRsSymp4} and \ref{prop:MaximalSubgroups}, 
\begin{align*}
P(x,t\th) &\leq \frac{4\cdot5(q_0+\sqrt{2q_0}+1)}{q(q-1)} \leq \frac{20(q+\sqrt{2q}+1)}{q(q-1)}
\end{align*}
and (ii) now follows. If $\th$ does not have order two, then $P(x,t\th) < \frac{1}{2}$ and (i) follows. (If $q=32$, then we use the observation that $q_0=2$ since $\th$ does not have order 2.) If $\th$ is an involutory graph-field automorphism, then, by Proposition~\ref{prop:MaximalSubgroups} (with $e=1$), and the refined bounds in Proposition~\ref{prop:FPRsSymp4}, 
\begin{align*}
P(x,t\th) &\leq \frac{8(q+\sqrt{2q}+1)}{q^2(q-1)} + \frac{1}{q^2} \leq \frac{16}{q^2} + \frac{1}{q^2} < \frac{18}{q^2}.
\end{align*}
Therefore, $u(G) \geq q^2/18$. This proves (i) and (iii), thus completing the proof.
\end{proof}

\subsection{Asymptotic results}\label{ssec:ProofAsymptotic}
Finally, let us turn to the remaining asymptotic results. Recall the notation for automorphisms which was introduced in Table~\ref{tab:Theta}. As intimated in Section~\ref{ssec:ProofDiagonal}, we now allow $\th \in \InnDiag(T)$. 
\begin{proposition}\label{prop:SympAsymN}
Let $q$ be odd and let $G = \< T, \th \>$ where $T=\PSp_{2m}(q)$ and $\th \in \{ \p^i, \d\p^i\}$. Then $u(G) \to \infty$ as $m \to \infty$.
\end{proposition}

\begin{proof}
We will follow the probabilistic approach but with a different choice of element $t\th$. Assume that $m$ is large enough so that $m > 5$ and there exists $d \in \Nat$ for which $\sqrt{2m}/8 < d < \sqrt{2m}/4$ and $(d,m-d)=1$. If $\th = \p^i$ then let $y = [A_{2d},A_{2m-2d}] \in \Sp_{2m}(q_0)$, and if $\th=\d\p^i$ then let $y = [C_{2d},C_{2m-2d}] \in \GSp_{2m}(q_0)$. By Proposition~\ref{prop:SympShintaniOdd}, let $t\th \in T\th$ such that $f(t\th) = \oly$. Therefore, a power of $f(t\th)$ lifts to an $X$-conjugate of $y$. \par

Let us now consider $\M(G,t\th)$. By Proposition~\ref{prop:ShintaniTransfer}, the unique $\C_1$ subgroup of $G$ containing $t\th$ has type $\Sp_{2d} \times \Sp_{2m-2d}$. There are at most $2m$ types of subgroup in the families $\C_2$, $\C_3$, $\C_4$, $\C_7$, and at most $e$ types of $\C_5$ subgroups. In each of these cases, by Proposition~\ref{prop:CentraliserBound}, there are at most $(q_0^d+1)(q_0^{d-m}+1) \leq 2q^{m/2}$ subgroups of each type in $\M(G,t\th)$. There are no $\C_6$ or $\C_8$ subgroups in $\M(G,t\th)$. Since $(d,m-d)=1$, a suitable power of $y$ is $z = [A_{2d},I_{2m-2d}]$. Observe that $z$ has a 1-eigenspace of codimension $2d < \sqrt{2m}/{2}$. Therefore, since $2m > 10$, by \cite[Theorem 7.1]{ref:GuralnickSaxl03}, $z$, and hence $t\th$, is not contained in any subgroups in the $\S$ family. (Exceptions involving the fully deleted permutation module do not occur since $p \neq 2$; see \cite[Table 2.1]{ref:Burness074}, for example.) \par

Using the bounds from Propositions~\ref{prop:FPRs} and \ref{prop:FPRsSubspace}, if $x$ has prime order, then \[ P(x,t\th) \leq \frac{2}{q^{m-2}} + \frac{1}{q^{m}} + \frac{1}{q^{\sqrt{2m}/8}} + \frac{1}{q^{2m-\sqrt{2m}/2}} + \frac{2(8m+e)(2q+2)^{1/2}}{q^{m/2-1}} \to 0 \] as $m \to \infty$. Therefore, $u(G) \to \infty$ as $m \to \infty$.
\end{proof}

We now turn to upper bounds on spread. In \cite[Prop. 2.5]{ref:GuralnickShalev03}, Guralnick and Shalev prove the following.
\begin{theorem}\label{thm:GuralnickShalevUpperBounds}
Let $m \geq 2$. 
\begin{enumerate}
\item{If $q$ is even and $T=\PSp_{2m}(q)$, then $s(T) \leq q$.}
\item{If $T=\Om_{2m+1}(q)$, then $s(T) < \frac{q^2+q}{2}$.}
\end{enumerate}
\end{theorem}

We now establish a generalisation of Theorem~\ref{thm:GuralnickShalevUpperBounds}.
\begin{proposition}\label{prop:AsymNegative}
Let $G = \< T, \th \> \in \A$.
\begin{enumerate}
\item{If $q$ is even, $T=\PSp_{2m}(q)$ and $\th$ is not a graph-field automorphism, then $s(G) \leq q$.}
\item{If $T=\Om_{2m+1}(q)$, then $s(G) < \frac{q^2+q}{2}$.}
\end{enumerate}
\end{proposition}

\begin{proof}
First consider (i). In the proof of \cite[Prop. 2.5(ii)]{ref:GuralnickShalev03}, a set $\mathcal{X}$ of $q+1$ transvections in $T$ is constructed with the property that for all subgroups $H_0$ of $T$ with type $\O^+_{2m}(q)$ or $\O^-_{2m}(q)$ there exists $x \in \mathcal{X}$ such that $x \in H_0$. Let $g \in G$. By Corollary~\ref{cor:OrthogonalSubgroups}, $G$ has at least one subgroup $H$ of type $\O^+_{2m}(q)$ or $\O^-_{2m}(q)$ such that $g \in H$. Therefore, there exists $x \in \mathcal{X}$ such that $x \in H$. As a result, $\< x, g \> \neq G$ and $s(G) \leq q$. \par

Now consider (ii). Let $V = \F_q^{2m+1}$ and consider the semilinear action of $G$ on $V$. Write $\ell = (q^2+q)/2$. In the proof of \cite[Prop. 2.5(i)]{ref:GuralnickShalev03}, a set $\mathcal{Y}$ of $\ell$ reflections in $T$ is constructed such that for all vectors $v \in V$ there exists $y \in \mathcal{Y}$ such that $vy=v$. Let $g \in G$. We will show that $g$ fixes a vector $v \in V$. If $g \in \InnDiag(T)$, then the set of eigenvalues of $g$ is closed under taking inverses. Therefore, an odd number of eigenvalues are a square root of unity. If all such eigenvalues are $-1$, then $\det(g)=-1$, which is a contradiction. So $g$ has a 1-eigenvector. If $g \in G \setminus \InnDiag(T)$, then $g$ is $G$-conjugate to the standard field automorphism. Therefore, there is a basis for $V$ consisting of vectors fixed by $g$. Thus $g$ fixes a vector, so there exists $y \in \mathcal{Y}$ such that $\< g, y \> \neq G$. Hence, $s(G) < \ell$.
\end{proof}

\subsection{Proof of main theorems}\label{ssec:ProofMainResults}
We now prove the four main theorems.
\begin{proof}[Proof of Theorems~\ref{thm:MainResult}--\ref{thm:MainUpper}]
Theorems~\ref{thm:MainResult} and \ref{thm:MainSharper} follow from Proposition~\ref{prop:InnDiag} (if $\th \in \InnDiag(T)$) and Propositions~\ref{prop:ProbSymp}--\ref{prop:ProbS4} (if $\th \in \Aut(T) \setminus \InnDiag(T)$). Theorem~\ref{thm:MainUpper}, and hence the forward implication of Theorem~\ref{thm:MainAsymptotic}, is a consequence of Proposition~\ref{prop:AsymNegative}. Therefore, it remains to verify the reverse implication of Theorem~\ref{thm:MainAsymptotic}. \par

Let $(G_i)$ be a sequence of groups in $\A$ with $|G_i| \to \infty$. Suppose that $(G_i)$ has no subsequence of odd-dimensional orthogonal groups or even characteristic symplectic groups, over a field of fixed size. Then $(G_i)$ is the union of at most three sequences: symplectic groups in odd characteristic with $q \to \infty$ or $n \to \infty$; symplectic groups in even characteristic with $q \to \infty$; and odd-dimensional orthogonal groups with $q \to \infty$. By Propositions~\ref{prop:InnDiag}, \ref{prop:ProbSymp}, \ref{prop:ProbOrth} and \ref{prop:SympAsymN}, the uniform spread of these sequences, so of the sequence $(G_i)$, diverges to infinity. This completes the proof of Theorem~\ref{thm:MainAsymptotic}.
\end{proof}


\end{document}